\documentclass[11pt]{amsart}
\usepackage{latexsym, amsmath, amssymb}

\usepackage{amsmath}
\usepackage{amsfonts}
\usepackage{amssymb}
\usepackage{amscd}
\usepackage{amsthm}
\usepackage{xcolor}
\usepackage{verbatim}
\usepackage{xypic}
\theoremstyle{plain}
\newtheorem{theoremint}{Theorem}
\newtheorem{theorem}{Theorem}[section]
\newtheorem{proposition}[theorem]{Proposition}
\newtheorem{lemma}[theorem]{Lemma}

\theoremstyle{definition}
\newtheorem{definition}[theorem]{Definition}
\newtheorem{remark}[theorem]{Remark}
\newtheorem{example}[theorem]{Example}

\newtheorem{conj}[theorem]{Conjecture}

\DeclareMathOperator{\Ext}{Ext}

\DeclareMathOperator{\Hom}{Hom}

\DeclareMathOperator{\coker}{coker}
\DeclareMathOperator{\ext}{ext}

\DeclareMathOperator{\hh}{h}
\DeclareMathOperator{\HH}{H}

\newcommand{\arr}{\longrightarrow}

\def\Z{\mathbb Z}

\newcommand{\zz}{{\mathbb Z}}

\newcommand{\cc}{{\mathbb C}}
\newcommand{\pp}{{\mathbb P}}

\newcommand\sC{{\mathcal C}}

\newcommand\sE{{\mathcal E}}
\newcommand\sF{{\mathcal F}}
\newcommand\sG{{\mathcal G}}

\newcommand\sI{{\mathcal I}}

\newcommand\sK{{\mathcal K}}
\newcommand\sL{{\mathcal L}}

\newcommand\sN{{\mathcal N}}
\newcommand\sO{{\mathcal O}}

\newcommand\sT{{\mathcal T}}

\newcommand\fD{{\mathfrak D}}

\def\pee#1{\hbox{$ {\mathbb P}^{#1}$}}

  \def \tab#1{\kern #1 truein}

\title[Instanton bundles on the flag variety F(0,1,2)]{Instanton bundles on the flag variety F(0,1,2)}
\author{F. Malaspina, S. Marchesi and J. Pons-Llopis}

\address{Politecnico di Torino, Corso Duca degli Abruzzi 24, 10129 Torino, Italy}
\email{francesco.malaspina@polito.it}

\address{Instituto de Matem\'{a}tica, Estat\'{i}stica e Computa\c{c}ao Cient\'{i}fica – UNICAMP, Rua S\'{e}rgio Buarque de Holanda 651, Distr. Bar\~{a}o Geraldo, 13083-859 Campinas, SP, Brazil}
\email{marchesi@ime.unicamp.br}

\address{DISIM
University of L'Aquila
Via Vetoio, Loc. Coppito
I-67100 L'Aquila, Italy}
\email{juanfrancisco.ponsllopis@univaq.it}

\keywords{Flag variety, instanton bundles, moduli spaces, jumping conics}
\subjclass[2010]{14F05;14J60}
  
\begin{document}

\maketitle
\begin{abstract}
Instanton bundles on $\mathbb{P}^3$ have been at the core of the research in Algebraic Geometry during the last thirty years. Motivated by the recent extension of their definition to other Fano threefolds of Picard number one, we develop the theory of instanton bundles on the complete flag variety $F:=F(0,1,2)$ of point-lines on $\mathbb{P}^2$. After giving for them two different monadic presentations, we use it to show that the moduli space $MI_F(k)$ of instanton bundles of charge $k$ is a geometric GIT quotient and the open subspace $MI^s_F(k)\subset MI_F(k)$ of stable instanton bundles has a generically smooth component of dim $8k-3$. Finally we study their locus of jumping conics.
\end{abstract}
\section*{Introduction}
One of the classes of vector bundles which has raised most of the attention among algebraic geometers is the class of \emph{instanton bundles}. They were defined as the algebraic counterpart that permitted to solve a central problem in Yang-Mills theory, a gauge theory for non-abelian groups whose aim is to provide an explanation of weak and strong interactions. It generalizes Maxwell's theory of electromagnetism, which can be rephrased as a gauge theory for the abelian group $U(1)$. In the original Yang-Mills theory, the term instanton or pseudo-particle denoted the minimum action solutions of the $SU(2)$ Yang-Mills equations in the $4$-sphere $S^4$. This problem could be rephrased in terms of differential geometry as follows: find all possible connections with self-dual curvature on a smooth $SU(2)$-bundle $\sE$ over the $4$-sphere $S^4$. Instantons come with a ``topological quantum number'' that can be interpreted as the second Chern class $c_2(\sE)$ of the $SU(2)$-bundle on $S^4$ and it is known as the ``charge'' of $\sE$.
A major input to obtain the classification of the set of solutions of the Yang-Mills field equations was given by twistor theory, as it was developed by R. Penrose. The essence of the twistor programme is to encode the differential geometry of a manifold by holomorphic data on some auxiliary complex space (a twistor space). In this setting, it was realized that the original Yang-Mills problem on (anti)self-dual $SU(2)$-connections on a bundle $\sE$ on $S^4$ could be restated in terms of the possible holomorphic structures on the pull-back bundle $\pi^{*}\sE$ on $\pp^3(\mathbb{C})$, where $\pi:\pp^3(\mathbb{C})\rightarrow S^4$ is the projection constructed through twistor theory. This, jointly with the characterization of holomorphic vector bundles on $\pp^3$ as a certain kind of monads due to Horrocks, led to a complete classification of instantons on $S^4$ (cf. \cite{AW}, \cite{AHDM}).

Motivated by the previous theorem, the notion of a \emph{mathematical instanton bundle} $\sE$ with charge $c_2(\sE)=k$ was defined on $\pp^3$  and more recently on an arbitrary \emph{Fano threefold} with Picard number one (cf. \cite{Fa2} and \cite{Kuz}). The existence of instanton bundles have been proved for almost all Fano threefolds of Picard number one. It became a central problem to understand the geometrical properties of their moduli spaces (smoothness, irreducibility, rationality...). These objects are important on their own as well as for the fact that physicists have interpreted various moduli spaces as solution spaces to physically interesting differential equations.

In this paper we propose to extend the study of instanton bundles on other Fano threefolds of higher Picard rank. In particular, we focus our attention to the case of the \emph{Flag variety} $F:=F(0,1,2)\subset\pp^7$ of pairs point--line in $\mathbb P^2$. The main reason for this choice is the following: in \cite{Hit}, Hitchin showed that the only twistor spaces of four dimensional (real) differential varieties which are K\"{a}hler (and \emph{a fortiori}, projective) are $\pp^3$ and the flag variety $F(0,1,2)$, which is the twistor space of $\pp^2$. Indeed, the study of instanton bundles on $F(0,1,2)$ was initiated in \cite{Bu} and \cite{Don}.

An \emph{instanton bundle} $\sE$ on $F$ with charge $k$  (or $k$-instanton) will be defined as a rank $2$ $\mu$-semistable vector bundle with  Chern classes $c_1(\sE)=0$, $c_2(\sE)=kh_1h_2$, $H^0(\sE)=0$ and satisfying the ``instantonic condition" $H^1(\sE(-1,-1))=0$ (see Definition (\ref{matheminst})). We are going to denote by $MI_F(k)$ the set parameterizing isomorphic classes of $k$-instantons. In this paper we show that $MI_F(k)$ has a pleasant geometric structure. For this, we are going to rely on their presentation as the cohomology of a certain monad: recall that, from \cite{Fa2} and \cite{Kuz},  when the derived category of coherent sheaves $D^b(X)$ on the Fano threefold $X$ has a full exceptional collection, instanton bundles on $X$ have a particular nice presentation as the cohomology of a certain monad. In general, this is not true for an arbitrary Fano threefold. Luckily, this is the case for the flag variety $F$ we are dealing with. Therefore we can give two such presentations (linked by a commutative diagram), each one involving members of well-chosen full exceptional collections  of $D^b(F)$. Concerning the first one, we obtain (Theorem \ref{secondmonad}):

\begin{theoremint}
  Let $\sE$ be an instanton bundle with charge $k$ on $F$. Then, up to permutation, $\sE$ is the cohomology of the monad
\begin{equation}\label{mon2} 0 \to \sO_F(-1,0)^{\oplus k}\oplus \sO_F(0,-1)^{\oplus k} \xrightarrow{\alpha}  \sG_1(-1,0)^{\oplus k}\oplus \sG_2(0,-1)^{\oplus k}
\xrightarrow{\beta} \sO_F^{\oplus 2k-2} \to 0.
\end{equation}
\noindent where $\sG_i$ is the pull-back of the twisted cotangent bundle $\Omega_{\pp^2}(2)$ from the two natural projections $p_i:F\subset\pp^2\times\pp^2\arr\pp^2$. Conversely, the cohomology of any such monad gives a $k$-instanton.
\end{theoremint}

\noindent The second description is not new; indeed it was already obtained in \cite[Section 3]{Bu}. In order to present it, let us introduce some notation. Let us denote by $U=H^0(\sO_F(1,0))$ the $3$-dimensional vector space of global sections and observe $U^*\cong H^0(\sO_F(0,1))$. Let $W$, $I_1$ and $I_2$ be vector spaces of dimension $4k+2$, $k$ and $k$ respectively. Let $J:W\rightarrow W^*$ be a nondegenerate skew-symmetric form and denote by $Sp(W,J)$ the symplectic group. Associated to any map of vector bundles $\tilde{A}:W\otimes\sO_F\rightarrow (I_1\otimes\sO_F(1,0))\oplus (I_2\otimes\sO_F(0,1))$ we have a map at the level of global sections:
$$
A:W\rightarrow (I_1\otimes U)\oplus (I_2\otimes U^*)
$$

and therefore a matrix (that we denote the same) $A\in W^*\otimes((I_1\otimes U)\oplus(I_2\otimes U^*))$. Let us define:
$$
\fD_k:=\{A\in W^*\otimes((I_1\otimes U)\oplus(I_2\otimes U^*))\mid AJA^t=0\},
$$
\noindent and
$$
\fD_k^0:=\{A\in \fD_k\mid  A \text{ is injective and } \tilde{A}:W\otimes\sO_F\rightarrow (I_1\otimes\sO_F(1,0))\oplus (I_2\otimes\sO_F(0,1))\text{  is a surjective bundle map}\}.
$$

\noindent The group $G_k:=Sp(W,J)\times GL(I_1)\times GL(I_2)$ acts on $\fD_k$ as follows: $(\eta,g_1,g_2)A=\binom{g_1\quad 0}{0\quad g_2}A\eta^t$ and this action clearly descends to $\fD_k^0$. Then, by using again derived category techniques we are able to recover the monad proposed by Buchdahl (see \cite{Bu}, display $(3.9)$, setting $n=2$ and $l=0$) and Donaldson (for charge $k=1$; see \cite{Don}, display $(14)$) associated to a $k$-instanton bundle $\sE$ on the flag $F$:

\begin{equation}\label{mon}
0 \to (I_1^*\otimes\sO_F(-1,0))\oplus(I_2^*\otimes\sO_F(0,-1))  \xrightarrow{J\tilde{A}^t}  W\otimes\sO_F
\xrightarrow{\tilde{A}} (I_1\otimes\sO_F(1,0))\oplus (I_2\otimes\sO_F(0,1))\to 0.
\end{equation}

\noindent  Conversely, the cohomology $\sE$ of such monad is a $k$-instanton bundle  whenever it has no global sections (notice that this is equivalent to $A$ being injective at the level of global sections). Therefore, the moduli space $MI_F(k)$ is the geometric quotient $\fD_k^0/G_k$.

Instanton bundles $\sE$ on $F$ will turn out to be Gieseker semistable.  In particular, we can talk about the open subspace  $MI_F^s(k)\subset MI_F(k)$ of stable $k$-instantons. $MI^s(k)$ can be identified with the open subspace of the Maruyama moduli space $M_F^s(2,0,kh_1h_2)$ of stable rank two bundles with those fixed Chern classes and $H^1(\sE(-1,-1))=0$. Concerning its non-emptiness and main properties we managed to prove by induction the following result (see Theorem \ref{existenceinst}):

\begin{theoremint}
  Let $F\subset\pp^7$ be the flag variety. The moduli space of stable $k$-instanton bundles $MI_F^s(k)$ has a generically smooth irreducible component of dimension $8k-3$.
\end{theoremint}
It is worth to observe that for a large family of Fano threefolds with Picard number one, Faenzi obtained an analogous result in \cite[Theorem A]{Fa2}. Moreover, in the case of
$\pp^3$, much more is known (see \cite{CO}, \cite{T1}, \cite{T2}, \cite{JV}): the moduli space of instantons of arbitrary charge is an affine irreducible smooth variety of dimension $8k-3$. In this paper, we prove the following (see Theorem \ref{affine-two}):
\begin{theoremint}\label{charge1}
The moduli space of $1$-instanton bundles $MI_F(1)$ is an affine irreducible smooth variety of dimension $5$.
\end{theoremint}

\bigskip

For an instanton bundle $\sE$ on the projective space $\pp^3$, the study of the behavior of the restriction of $\sE$ to the lines has played a crucial role. On the flag threefold $F$, we have realized that an analogous attention should be devoted to the study of the restriction of an instanton bundle to the conics. Therefore, we give a proof of the fact that the Hilbert scheme $H:=Hilb^{2t+1}(F)$ parameterizing rational curves of genus $0$ and degree $2$ is again isomorphic to $\pp^2\times\pp^2$. Under this setting, if we denote by $\mathcal D_\sE$ the set of curves on $H$ for which the restriction of $\sE$ is not trivial (the "jumping conics"), we obtain the following result (see Prop. \ref{jumpingdivisor}):

\begin{theoremint}
  Let $\sE$ be a $k$-instanton on $F$. Then $\mathcal D_{\sE}$ is a divisor of type $(k,k)$  equipped with a torsion-free sheaf
$G$
fitting into
\begin{equation}\label{jump}0 \to\sO_H(-1,-1)^{\oplus k}\oplus \sO_H(-1,0)^{\oplus k}\to  \sO_H^{\oplus k}\oplus \sO_H(-1,0)^{\oplus k}\to i_*G\to  0.
\end{equation}
\end{theoremint}

This paper is organized in the following way: in the next section, we recall the definition and basic facts of the flag variety $F:=F(0,1,2)$. In Section 2, we introduce the notion of instanton bundle on $F$. In Sections 3 and 4, we study the derived category $D^b(F)$ of coherent sheaves on $F$ and use it to give a monadic presentation of the instanton bundles.     The central part of this paper is Section 5, where we prove the existence of a suitable family of instantons as stated in Theorem 1. Finally, in Section 6, we deal with the notion of jumping conic on the flag variety.

\noindent\textbf{Acknowledgements.}
The authors wish to thank D. Faenzi and G. Sanna for fruitful discussions. The second and third authors would like to thank the Politecnico di Torino and Research Institute for Mathematical Sciences in Kyoto, where parts of this project were developed, for the hospitality and for providing the best working conditions. The authors also want to thank the referees for a careful reading that leaded to an improvement of the paper.

The second author was partially supported by Funda\c{c}\~{a}o de Amparo \`{a} Pesquisa do Estado de S\~{a}o Paulo (FAPESP) grant 2017/03487-9 and by CNPq grant 303075/2017-1. The third author was supported by a JSPS Postdoctoral Fellowship.

\section{Preliminaries}
 We will work over an algebraically closed field $K$ of characteristic zero.

 Let $F\subseteq\mathbb P^7$ be the del Pezzo threefold  of degree $6$, we can also construct $F$ as the general hyperplane section of $\mathbb P^2\times\mathbb P^2$. The projections $\pi_i$ induce maps $p_i\colon F\to\mathbb P^2$ by restriction, $i=1,2$ and such maps are isomorphic to the canonical map $\mathbb P(\Omega_{\mathbb P^2}^1(2))\to\mathbb P^2$. Thinking of the second copy of $\mathbb P^2$ as the dual of the first one, then $F$ can also be viewed naturally as the flag variety of pairs point--line in $\mathbb P^2$. We denote by $A(F)$ the Chow ring of $F$. Let  $h_i$, $i=1,2$, be the respective classes of $p_i^*\sO_{\mathbb P^2}(1)$ in $A^1(F)$. The class of the hyperplane divisor on $F$ is $h=h_1+h_2$.
 For coherent sheaves $\sE,\sF$ on $F$ we are going to denote the twisted sheaf $\mathcal{E}\otimes\mathcal{O}_F(\alpha h_1 + \beta h_2)$ by $\mathcal{E}(\alpha,\beta)$. As usual, $H^i(X,\mathcal{E})$ stands for the cohomology groups, $h^i(X,\mathcal{E})$ for their dimension and analogously $\ext(\sE,\sF)$ will denote the dimension of the $\Ext(\sE,\sF)$ group under consideration.

 The above discussion proves the isomorphisms
$$
A(F)\cong A(\mathbb P^2)[h_1]/(h_1^2-h_1h_2+h_2^2)\cong \mathbb Z[h_1,h_2]/(h_1^2-h_1h_2+h_2^2, h_1^3,h_2^3).
$$
In particular, $Pic(F)\cong\mathbb Z^{\oplus2}$ with generators $h_1$ and $h_2$. We will denote the Chern polynomial of a given coherent sheaf $\sE$ by
$$
c_{\sE}(t):=1 + (\alpha_1 h_1 + \alpha_2 h_2)t + (\beta_1 h_1^2 + \beta_1 h_2^2)t^2 + \gamma h_1^2 h_2t^3,
$$
\noindent where the coefficient of degree $i$ is $c_i(\sE)$.
\\

Recall that $F$ contains two families of lines $\Lambda_1$,$\Lambda_2$, each isomorphic to $\pp^2$. Their representatives in the Chow ring $A(F)$ are $h_1^2$, $h_2^2$. Notice that if we look at $F$ as the projective bundle $\mathbb P(\Omega_{\mathbb P^2}^1(2))\to\mathbb P^2$, these families correspond to the fibers over points of $\pp^2$. We have a geometrical description (using the notion of flag variety): given $p\in\pp^2$, $\lambda_p:=\{L\in\pp^{2\vee}\mid p\in L\}\in\Lambda_1$. Analogously, given a line $L\subset\pp^2$, $\lambda_L:=\{x\in\pp^2\mid x\in L \}\in\Lambda_2$. Notice $\lambda_x\cap\lambda_y=\emptyset$ if $x\neq y$ (clear from cohomological product $h_1^2h_2^2=0$) and $\lambda_x\cap\lambda_L=\emptyset$ (resp. $\{x,L\}$) if $x\in l$ (resp. $x\notin L$). If $L_1$ (resp. $L_2$) is a line from the family $\Lambda_1$ (resp. $\Lambda_2$) it holds that
$$
\sO_F(\alpha,\beta)\otimes\sO_{L_1}\cong\sO_{\pp^1}(\beta) \quad\text{(resp.}\sO_F(\alpha,\beta)\otimes\sO_{L_2}\cong\sO_{\pp^1}(\alpha))
$$

\noindent since $h_1^2(\alpha h_1+\beta h_2)=\beta h_1^2h_2$. The $\sO_F$-resolutions of a line $L_1$ is:

\begin{equation}\label{res-line}
  0 \longrightarrow \sO_F(-2, 0) \longrightarrow  \sO_F(-1,0)^2 \longrightarrow \sO_F \longrightarrow \sO_{L_1} \longrightarrow 0;
\end{equation}
\noindent (or the analogous one for the second family of lines $L_2$).

In order to compute the $\sO_F$-resolution of a point $p\in F$, we can consider its $\sO_L$-resolution

$$
 0 \longrightarrow \sO_L(-1) \longrightarrow  \sO_L \longrightarrow \sO_p\longrightarrow 0
$$
and use the mapping cone construction to conclude that

\begin{equation}\label{res-point}
   \begin{array}{rccccl}
& \sO_F(-2,0) & &\sO_F(-1,0)^2  \\
0 \longrightarrow \sO_F(-3,-1) \longrightarrow  & \oplus & \longrightarrow  & \oplus & \longrightarrow & \sO_F \longrightarrow \sO_p \longrightarrow 0.\\
& \sO_F(-2,-1)^2 & & \sO_F(-1,-1)
\end{array}\end{equation}



The flag variety $F$ also contains a family of conics $C$ (see the next subsection) whose $O_F$-resolution is:

\begin{equation}\label{res-conic}
  \begin{array}{rcl}
& \sO_F(-1,0) \\
0 \longrightarrow \sO_F(-1,-1) \longrightarrow  & \oplus & \longrightarrow \sO_F \longrightarrow \sO_C \longrightarrow 0.\\
& \sO_F(0,-1)
\end{array}
\end{equation}

\noindent We have to distinguish two different cases (see Remark \ref{rmk-hilbert-conics}). In the case of a smooth conic $C\cong\pp^1$, it holds:
$$
\sO_F(\alpha,\beta)\otimes\sO_{C}\cong\sO_{\pp^1}(\alpha+\beta)
$$
which will be denoted either by $\sO_C(\alpha,\beta)$ or $\sO_C(\alpha+\beta)$, to remember we are restricting at a conic.\\
In the case of a reducible conic $C = L_1 \cup L_2$, we will always use the notation $\sO_C(\alpha,\beta)$ to keep track of the degree on each one of the lines.

We now compute the Hirzebruch-Riemann-Roch formula for a coherent sheaf $\sE$ of rank $r$ on the flag variety $F$.
Recall, that denoting by $T_F$ the tangent bundle of the flag, we have
$$
c_1(T_F) = 2 (h_1 + h_2), \:\:\:\: c_2(T_F)=6h_1h_2
$$
and we compute the Todd polynomial
$$
\begin{array}{cl}
td(T_F)  & = 1 + \frac{1}{2}c_1 + \frac{1}{12}(c_1^2 + c_2) + \frac{1}{24} c_1 c_2 =\vspace{0.3cm}\\
& = 1 + (h_1 + h_2) + \frac{3}{2} h_1h_2 + \underbrace{1}_{=h_1^2h_2}
\end{array}
$$
Remembering that the Chern character is given by
$$
ch(\sE) = r + c_1 +\frac{1}{2}(c_1^2 - 2c_2) + \frac{1}{6}(c_1^3 - 3 c_1c_2 + 3c_3)
$$
we can finally obtain
\begin{equation}\label{eulerchar}
\chi(\sE) = r + \frac{3}{2} c_1 h_1 h_2 + \frac{1}{2}(c_1^2 - 2 c_2)(h_1+h_2) + \frac{1}{6}(c_1^3 - 3 c_1c_2 + 3c_3)
\end{equation}

In particular, for a rank $2$ vector bundle $\sE$ with Chern classes $c_1(\sE)=0$ and $c_2(\sE)=kh_1h_2$ it holds:

\begin{equation}\label{dim-chi-end}
  \chi(\sE\otimes\sE^{\vee})=4-8k.
\end{equation}
\noindent This formula will be useful later on to understand the dimension of moduli spaces of instanton bundles. Moreover, we will need to understant how the Chern classes
of a coherent sheaf vary under twists. This is provided by the following Lemma (see \cite[Prop. 5.17]{Eisenbud-Harris}, where it is stated for vector bundles, but it can be extended to arbitrary coherent sheaves because of the universal nature of such formulas, as in \cite[Lemma 2.1]{Har}).

\begin{lemma}\label{twist-chern-classes}
  Let $\sE$ be a coherent sheaf of rang $r$ on the flag variety $F$ and let $\sL$ be a line bundle. Then
  $$
  c_k(\sE\otimes\sL)=\sum_{i=0}^{k}\binom{r-k+i}{i}c_{1}(\sL)^ic_{k-i}(\sE).
  $$

\end{lemma}

Gathering the previous information, we can compute the Chern polynomial of the coherent sheaf $\sO_C(1)$, $C$ being a conic on $F$, which will be the right candidate to use the induction process explained in Section \ref{sec-ind}. First of all, using that the Chern polynomial is multiplicative on exact sequences, we obtain from the resolution (\ref{res-conic}) of $\sO_C$ that
$$
c_{\sO_C}(t)=1-(h_1^2+h_2^2)t^2+2h_1^2h_2t^3.
$$
\noindent Therefore, applying Lemma \ref{twist-chern-classes}, we obtain $c_{\sO_C(1)}(t)=1-(h_1h_2)t^2$.

Finally, we want to understand the Chern classes of a reflexive coherent sheaf $\sE$ (i.e., $\sE^{\vee\vee}\cong\sE$) of rank two on $F$. The case of $\pp^3$ has been studied in \cite{Har} but the results there can be easily adapted to our context (or, as a matter of fact, to an arbitrary smooth threefold $X$). So let $\sE$ be a rank two reflexive sheaf on $F$. By the Auslander-Buchsbaum formula, $\sE$ has homological dimension one, namely, there exists a resolution by locally free sheaves of the form:

\begin{equation}\label{locfreeres}
  0\rightarrow\sF_0\rightarrow\sF_1\rightarrow\sE\rightarrow 0.
\end{equation}
Taking duals, we have:

\begin{equation}\label{locfreeres2}
  0\rightarrow\sE^{\vee}\rightarrow\sF_1^{\vee}\rightarrow\sF_0^{\vee}\rightarrow \mathcal{E}xt^1(\sE,\sO_F)\rightarrow 0.
\end{equation}
\noindent $\mathcal{E}xt^1(\sE,\sO_F)$ is a coherent sheaf supported on the finite set of $n$ points where $\sE$ is not locally free. We need therefore the following

\begin{lemma}
  Let $\sG$ be a coherent sheaf on $F$ supported on a finite set of points of length $n$. Then $c_{\sG}(t)=1+2nt^3$.
\end{lemma}
\begin{proof}
  By induction on the length $n$ and multiplicity of the Chern polynomial, we can suppose that $\sG\cong\sO_p$ for a point $p\in F$. But then a straightforward computation, using the $\sO_F$-resolution (\ref{res-point}) of $\sO_p$, gives the statement. Alternatively, we could use the formulas from Lemma \ref{twist-chern-classes}, the Hirzebruch-Riemann-Roch formula recalled in (\ref{eulerchar}) and the fact that $\sG$ does not change under twist by a line bundle to see that $c_1(\sG)=c_2(\sG)=0$ and
  $$
  h^0(\sG)=\chi(\sG)=\frac{1}{2}c_3(\sG).
  $$
\end{proof}

\noindent Therefore, combining the previous Lemma, the multiplicity of the Chern polynomials on the exact sequences (\ref{locfreeres}) and (\ref{locfreeres2}) and the fact that for a locally free sheaf $\sF$ it holds $c_{\sF^{\vee}}(t)=c_{\sF}(-t)$ we obtain:
$$
c_{\sE^{\vee}}(t)=c_{\sE}(-t)(1+2nt^3).
$$
\noindent To conclude, from Lemma \ref{twist-chern-classes} and the fact that $\sE^{\vee}\cong\sE(-c_1(\sE))$, we obtain the following lemma:

\begin{lemma}\label{Chern-reflexive-flag}
  Let $\sE$ be a reflexive coherent sheaf of rank two on $F$. Then $c_3(\sE)=h^0(\mathcal{E}xt^1(\sE,\sO_F))\geq 0$ counts the number of points for which $\sE$ is not locally free. In
  particular, $c_3(\sE)=0$ if and only if $\sE$ is a locally free sheaf.
\end{lemma}



We will now recall how to compute the cohomology of the line bundles on $F$ (see \cite{CFM3} Proposition 2.5):
 \begin{proposition}
\label{pLineBundle}
For each $\alpha_1,\alpha_2\in\mathbb Z$ with $\alpha_1\le \alpha_2$, we have $$
h^i\big(F,\sO_F(\alpha_1,\alpha_2)\big)\ne0
$$
if and only if
\begin{itemize}
\item $i=0$ and $\alpha_1\ge0$;
\item $i=1$ and $\alpha_1\le -2$, $\alpha_1+\alpha_2+1\ge0$;
\item $i=2$ and $\alpha_2\ge0$, $\alpha_1+\alpha_2+3\le0$;
\item $i=3$ and $\alpha_2\le -2$.
\end{itemize}
In all these cases
$$
h^i\big(F,\sO_F(\alpha_1,\alpha_2)\big)=(-1)^i\frac{(\alpha_1+1)(\alpha_2+1)(\alpha_1+\alpha_2+2)}{2}.
$$
\end{proposition}

\subsection{Hilbert scheme of lines and conics on the Flag variety}

It could be thought that the study of the geometry of lines of the $F$ will be important to define and understand instanton bundles, as it turned out to be in the case of instanton bundles on $\pp^3$. Nevertheless, in the case of the flag variety, the main kind of rational curve we are interested in is the conic. In fact, through the Ward correspondence, instanton bundles on $F$ have trivial splitting on "real" conics (this is explained in \cite{Bu} and \cite{Don} without explicitly mentioning the degree) and therefore, by semicontinuity, on the general element of $H:=Hilb^{2t+1}(F)$. Therefore, we devote this subsection to study the main properties of the conics on $F$.

\begin{lemma}
The Hilbert scheme of rational curves of degree two $H:=Hilb^{2t+1}(F)$ is isomorphic to $\pee2\times\pee2$. The open set $\pee2\times\pee2\backslash F$ corresponds to smooth conics. Moreover, the canonical map $p:\sC\rightarrow F$ from the universal conic $\sC$ to $F$ endows $\sC$ with the structure of a quadric bundle of relative dimension $2$ over $F$.
\end{lemma}
\begin{proof}
  Quadric surfaces $Q\cong\pee1\times\pee1\subset\pee2\times\pee2\subset\pee8$ are parameterized again by a $\pee2\times\pee2$. Indeed, in order to define a quadric, we need to choose a pair of lines $(L_1,L_2)\in\mathbb{P}^{2 \vee}\times\mathbb{P}^{2 \vee}\cong\pee2\times\pee2$. $Q_{(L_1,L_2)}:=L_1\times L_2\subset\pee2\times\pee2$ is the associated quadric. Now, it is immediate to see that any quadric in $\pee2\times\pee2$ defines uniquely a conic $C_{(L_1,L_2)}:=F\cap Q_{(L_1,L_2)}$ in the hyperplane section $F$. Moreover, one can check that $Q_{(L_1,L_2)}$ is tangent to $F$ (and therefore $C_{(L_1,L_2)}$ is singular) if and only if $(L_1,L_2)\in F\subset\pee2\times\pee2$.
\end{proof}

\begin{remark}\label{rmk-hilbert-conics}
  Indeed, it is known (see for instance \cite[Lemma 2.1.1]{KPS}) that any subscheme of $F$ with Hilbert polynomial $2t+1$ will be a smooth conic, a pair of distinct lines intersecting on a point, or a line with a double structure. In order to see that there is no such non-reduced subscheme on $F$ we should observe that for any line $L$ on $F$, we have $\sN_{L|F}\cong\sO_L^2$. Therefore there is no surjective map $\sN_{L|F}^{\vee}\to\sO_{L}(-1)$ and we conclude again by \cite[Lemma 2.1.1]{KPS}.
\end{remark}

\begin{remark}
  An explicit parametrization of the conics is as follows, for $(p,L)\in\pee2\times\mathbb{P}^{2 \vee}\backslash F$, the subscheme of F:
  $$
  C_{(p,L)}:=\{(q,S)\mid L(q)=S(p)=0\}
  $$
  is a smooth conic in $F$. Its class in the Chow ring is $h_1h_2$. On the other hand, if we choose a point $(p,l)\in F$ the associated curve $C_{(p,l)}$ will be the union of the lines $L_{p}:=\{(p,S)\mid s(p)=0\}\in h_1$ and $L_l:=\{(q,l)\mid l(q)=0\}\in h_2$ intersecting at the point $(p,l)$.
\end{remark}

One more argument to explain why the role of the line is taken over by conic in the flag is the following
\begin{proposition}
  Given two generic points of $F$, there exists exactly one smooth conic passing through them.
\end{proposition}

\begin{proof}
  Given two points $(p_i,L_i)\in F$, $i=1,2$, if we define $q:=L_1\cap L_2$, $S:=\overline{p_1,p_2}$, the conic $C_{(q,S)}$ meets the requirements.
\end{proof}
\begin{remark}
 Notice that if we allow singular conics, two more cases appear.
\end{remark}

\section{Instantons. definition and first properties}
In this section we will define the notion of instanton bundle on the flag variety $F$. Such definition is given through requirements on cohomology vanishing, trivial splitting and fixed Chern classes. As we will see in Section \ref{sec:monad}, such characterization does not lead to a \emph{linear monad}, i.e. a monad involving only the structure and hyperplane bundles of the flag variety.

\begin{definition}\label{matheminst}
For any integer $k\geq 1$ we will call a \emph{mathematical instanton bundle} with charge $k$  (or, for short, a $k$-instanton) a rank two $\mu$-semistable bundle $\sE$ on $F$ with $H^0(\sE)=0$, $c_1(\sE)=(0,0)$, $c_2(\sE)=kh_1h_2$ and $H^1(\sE(-1,-1))=0$. Its Hilbert polynomial is

\begin{equation}\label{hilbertpolyn}
\chi(\sE(t,t))= (t+1)(2t^2+4t+2(1-k)).
\end{equation}
\end{definition}

The interest in (mathematical) instanton bundles is rooted in a crucial problem in Yang-Mills theory (see  \cite{Bu}, \cite{Don} for details). It turns out that the flag variety $F$ is the twistor space of the projective plane $\pp^2$. The fibers of the map defining the twistor structure $F\rightarrow \pp^2$ are conics (namely smooth rational curves of degree $2$ with respect to $\sO_F(1,1)$) that cover the whole of $F$. We are going to call them \emph{real conics}.  As in the case of $\pp^3$, there is a one-to-one correspondence between the anti-self-dual solutions of the Yang-Mills equations over $\pp^2(\mathbb{C})$ and vector bundles $\sE$ on $F$ such that (see \cite[Theorem 2]{Bu}):
\begin{itemize}
  \item they split trivially on real conics: $\sE_{\mid C}\cong \sO_C\oplus\sO_C$ for any real conic $C\in Hilb^{2t+1}(F)$;
  \item they admit an isomorphism $\phi:\sE\rightarrow\sigma^*(\overline{\sE}^{\vee})$, where $\sigma:F\rightarrow F$ is an anti-holomorphic involution (the \emph{reality structure} on $\sE$).
\end{itemize}
\noindent The vector bundles on $F$ satisfying the two previous conditions are called \emph{real instanton bundles}.

\begin{remark}\label{remarkglobalsections}
\begin{enumerate}
  \item[i)] For Fano threefolds of Picard number one an analogous definition, extending the case of the projective space $\pp^3$, was given in \cite{Fa2} and \cite{Kuz}. In the case of the flag variety $F$, however, we have to allow properly $\mu$-semistable instanton bundles to deal with higher Picard number. We are going to characterize in a moment the properly $\mu$-semistable instantons on $F$.
  \item[ii)] The presentation we gave of the Hilbert polynomial of $\sE$ shows that $\chi(\sE(-1,-1))=0$. Indeed, using Serre's duality, it holds $H^i(F,\sE(-1,-1))=0$ for $i=0,\dots,3$.
  \item[iii)] It can be easily proven, following the computations in \cite[page 176]{OSS}, that Gieseker and $\mu$-stability are equivalent for our bundles.
  \item[iv)] It is worthwhile to point out that the condition $H^0(\sE)=0$ does not follow from the rest of the conditions defining an instanton bundle. Indeed, if we take the smooth curve $C:=\bigcup_{i=1}^k L_i\cup \bigcup_{i=1}^k L_i'$ for $L_i$ (resp. $L_i'$) $k$ lines from the first (resp. the second) family of lines on $F$ such that $L_i\cap L_j'=\emptyset$ for $i\neq j$, Serre construction gives a vector bundle $\sE$ sitting on the exact sequence
      $$
      0\rightarrow\sO_F\rightarrow \sE\rightarrow\sI_{C\mid F}\rightarrow 0,
      $$
      with $H^0(\sE)\neq 0$ but satisfying the rest of conditions for defining a $k$-instanton bundle.
\end{enumerate}
\end{remark}

Mathematical instanton bundles have trivial splitting type on generic conics.

\begin{proposition}\label{trivial-conic}
  Let $\sE$ be a $k$-instanton on $F$. Then for a generic conic $C\in M:=Hilb^{2t+1}(F)$, $\sE_C\cong\sO_C^2$.
\end{proposition}

\begin{proof}
We are going to use the main result from \cite{Hir}. Since the three hypothesis of the main Theorem from that paper are satisfied, we can conclude that for a general (smooth) conic $C\in Hilb^{2t+1}(F)$, if $\sE_C\cong\sO_C(-i)\oplus\sO_C(i)$ then $2i\leq -\mu_{min}(\sT_{\sC \mid F}\mid C)$ where $\sT_{\sC\mid F}$ is the relative tangent bundle of the canonical projection $p:\sC\rightarrow F$ and $\mu_{min}$ is the minimal slope of the Harder-Narasimhan filtration of its restriction to $C$. Again, by \cite[(6.3)]{Hir}, $\sT_{\sC \mid F}\mid C$ can be identified with the kernel of the evaluation map $\phi:H^0(C,\sN_{C\mid F})\otimes\sO_C\rightarrow \sN_{C\mid F}$ where $\sN_{C\mid F}$ is the normal bundle of the generic conic $C$ in $F$. Now, as we see from the resolution (\ref{res-conic}), since $C$ is the complete intersection of divisors of type $\sO_F(1,0)$ and $\sO_F(0,1)$, it turns out that $\sN_{C\mid F}\cong \sO_C(1)^2$ and therefore the kernel of the evaluation map $\phi$ is $\sO_C(-1)^2$. Therefore, $i=0$ and $C$ has trivial splitting type.

\end{proof}


In the next proposition, we show that the two traditional notions of (semi-)stability coincide for an instanton bundle and we get the following characterization of properly semistable instanton bundles; but first, in order to check the (semi)-stability of rank $2$ real instanton bundles, let us recall the following version of the Hoppe's criterion (see \cite{JMPS}):
\begin{lemma}[Hoppe's criterion]
A rank $2$ vector bundle $\sE$ on $F$ with first Chern class $c_1(\sE)=0$ is $\mu$-(semi)stable if and only if $H^0(F,\sE(p,q))=0$ for any $p,q\in \zz$ such that $p+q\leq (<) 0$.
\end{lemma}

\begin{proposition}\label{gieseker}
Let $\sE$ be a $k$-instanton on $F$. Then, it is also Gieseker semistable. Moreover, if $\sE$ is not $\mu$-stable, then $k=l^2$ for some $l\in \zz,l\neq 0$ and it can be constructed as an extension $\Lambda_l$ of the form

\begin{equation}\label{stabquad}
0\to\sO_F(l,-l)\to \sE\to \sO_F(-l,l)\to 0.
\end{equation}

\noindent The only common element of the two families of extensions $\Lambda_l$ and $\Lambda_{-l}$ is the decomposable bundle $\sO_F(l,-l)\oplus\sO_F(-l,l)$.
\end{proposition}

\begin{proof}

Suppose that the $k$-instanton bundle $\sE$ is not $\mu$-stable.  We deduce that there exists a positive integer $l$ such that $H^0(\sE(l,-l))\not=0$ and that $\sE$ sits on an exact sequence of the form:
\begin{equation}\label{aaaaa}
  0\to\sO_F\to \sE(l,-l)\to \sI_Y(2l,-2l)\to 0,
\end{equation}

where $Y\subseteq F$. Given that $H^0(F,\sE(l-1,-l))=H^0(F,\sE(l,-l-1))=0$, $Y$ is purely $2$-codimensional or empty. In order to see that the we are actually dealing with the latter case, notice that, for $C\subset F$ a conic, $C\cap Y=Z$ with $z:=length(Z)$, the restriction of the previous exact sequence to $C$ would be of the form
$$
0\to\sO_C\to \sE(l,-l)_{\mid C}\to \sO_Z\oplus\sO_C(-z)\to 0,
$$
or, in other words, the splitting of $\sE$ on $C$ is of the form $\sE_C\cong\sO_C(-z)\oplus\sO_C(z)$. Namely, $\sE$ has non-trivial splitting type on $C$ exactly when the intersection $C\cap Y$ is non-empty. On the other hand, notice that $Y$ is represented in the Chow ring $A(F)$ by $(k-l^2)h_1h_2$ and therefore the divisor on the Hilbert space $H=\pp^2\times\pp^2$ of conics intersecting $Y$ is of type $\sO_H(k-l^2,k-l^2)$. Since we will see (cf. Proposition \ref{jumpingdivisor}) that the set of jumping conics of a $k$-instanton bundle is of type
$\sO_H(k,k)$ we conclude that $l=0$ and get a contradiction with the fact that $H^0(\sE)=0$. Therefore $Y=\emptyset$ and $\sE$ sits on a short exact sequence of the announced form
\begin{equation}\label{stabquad}
0\to\sO_F(l,-l)\to \sE\to \sO_F(-l,l)\to 0,
\end{equation}
\noindent with $l^2=k\neq 0$. So $\sE$ is also Gieseker semistable. To prove the last assertion, let us suppose that there exists a $l^2$-instanton $\sE$ that fits at the same time in the extensions $\Lambda_l$ and $\Lambda_{-l}$. Then easily we could construct maps:
$$
\sO(-l,l)\oplus\sO(l,-l)\stackrel{(\alpha  \beta)}{\longrightarrow}\sE\stackrel{\binom{\delta}{\gamma}}{\longrightarrow}\sO(-l,l)\oplus\sO(l,-l)
$$
\noindent such that $\binom{\delta}{\gamma}(\alpha \beta)=\binom{Id\quad 0}{0 \quad Id}$ and therefore $\sE\cong\sO_F(l,-l)\oplus\sO_F(-l,l)$.

\end{proof}

\begin{lemma}\label{simple}
  Any $k$-instanton $\sE$ vector bundle on $F$ distinct from $\sO_F(l,-l)\oplus\sO_F(-l,l)$ is simple, namely, $End(\sE)=\mathbb{C}$. In particular, they carry an unique symplectic structure $\phi:\sE\rightarrow\sE^{\vee}$, $\phi^{\vee}=-\phi$.
\end{lemma}

\begin{proof}
  Since it is well-known that stable vector bundles are simple, let $\sE$ be a properly semistable instanton, $\sE\ncong\sO_F(l,-l)\oplus\sO_F(-l,l)$. Then it fits in an exact sequence of the form $\Lambda_l$ for a unique $0 \neq l\in\zz$. Tensoring it with $\sE$ and using that we have just seen that $h^0(F,\sE(-l,l))=1$ and $h^0(F,\sE(l,-l))=0$ we get, since $\sE\cong \sE^{\vee}$, $End(\sE)=H^0(F,\sE\otimes\sE)=\cc$. For the last statement, any nonzero two-form $\omega\in\bigwedge^2\sE\cong\sO$ defines equivalent symplectic structures on $\sE$.
\end{proof}

We are going to end this section showing that real instanton bundles are indeed mathematical instantons.

\begin{lemma}
  A real instanton bundle on $F$ is a mathematical instanton.
\end{lemma}
\begin{proof}
Let $\sE$ be a real instanton bundle on $F$. The condition $c_1(\sE)=0$ is clearly satisfied. On the other hand, it is proved in \cite[Lemma 4]{Bu} that $H^1(\sE(-1,-1))=0$ and also $c_2(\sE)=kh_1h_2$ for $k\geq 1$.  To conclude, let us show that $\sE$ is $\mu$-semistable, using the Hoppe's criterion: if $s\in H^0(F,\sE(p,q))$ with $p+q<0$, for $C\subset F$ an arbitrary real conic, $s_{\mid C}$ is a global section of $\sE_{\mid C}(p,q)\cong\sO_C(p+q)^2$, i.e., $s_{\mid C}=0$. Since $F$ is covered by real conics we see that $s$ should be the zero section and therefore we are done. Finally, the proof of $H^0(\sE)=0$ follows the lines of the proof of Proposition \ref{gieseker}. Indeed, a non-zero section of $\sE$ would give a short exact sequence of the form (\ref{aaaaa}) with $l=0$ and $Y$ a codimension two subscheme. But then, tensoring this short exact sequence with the structure sheaf $\sO_C$ for $C$ a real conic such that $Y\cap C\neq \emptyset$ would contradict that $\sE$ has trivial splitting type on $C$. Indeed, notice that in this argument the crucial point is that $F$ is \emph{covered} by conics with trivial splitting type.

\end{proof}

\section{The derived category of the flag variety $F=F(0,1,2)$}\label{sec:dercat}
In this section we will recall how to mutate exceptional collections and the Beilinson results, both needed to see the instanton bundles as cohomology of a  certain monad.

Given a smooth projective variety $X$, let $D^b(X)$ be the the bounded derived category of coherent sheaves on $X$.
An object $E \in D^b(X)$ is called exceptional if $\Ext^\bullet(E,E) = \mathbb C$.
A set of exceptional objects $E_1, \ldots, E_n$ is called an exceptional collection if $\Ext^\bullet(E_i,E_j) = 0$ for $i > j$.
An exceptional collection is full when $\Ext^\bullet(E_i,A) = 0$ for all $i$ implies $A = 0$, or equivalently when $\Ext^\bullet(A, E_i) = 0$ for all $i$ also implies $A = 0$.

\begin{definition}\label{def:mutation}
Let $E$ be an exceptional object in $D^b(X)$.
Then there are functors $\mathbb L_{E}$ and $\mathbb R_{E}$ fitting in distinguished triangles
$$
\mathbb L_{E}(T) 		\to	 \Ext^\bullet(E,T) \otimes E 	\to	 T 		 \to	 \mathbb L_{E}(T)[1]
$$
$$
\mathbb R_{E}(T)[-1]	 \to 	 T 		 \to	 \Ext^\bullet(T,E)^* \otimes E	 \to	 \mathbb R_{E}(T)	
$$
The functors $\mathbb L_{E}$ and $\mathbb R_{E}$ are called respectively the \emph{left} and \emph{right mutation functor}.
\end{definition}

The collections given by
\begin{align*}
E_i^{\vee} &= \mathbb L_{E_0} \mathbb L_{E_1} \dots \mathbb L_{E_{n-i-1}} E_{n-i};\\
^\vee E_i &= \mathbb R_{E_n} \mathbb R_{E_{n-1}} \dots \mathbb R_{E_{n-i+1}} E_{n-i},
\end{align*}
are again full and exceptional and are called the \emph{right} and \emph{left dual} collections. The dual collections are characterized by the following property; see \cite[Section 2.6]{GO}.
\begin{equation}\label{eq:dual characterization}
\Ext^k(^\vee E_i, E_j) = \Ext^k(E_i, E_j^\vee) = \left\{
\begin{array}{cc}
\mathbb C & \textrm{\quad if $i+j = n$ and $i = k $} \\
0 & \textrm{\quad otherwise}
\end{array}
\right.
\end{equation}

\begin{theorem}[Beilinson spectral sequence]\label{thm:Beilinson}
Let $X$ be a smooth projective variety with a full exceptional collection $\langle E_0, \ldots, E_n\rangle$ of objects for $D^b(X)$. Then for any $A$ in $D^b(X)$ there is a spectral sequence
with the $E_1$-term
\[
E_1^{p,q} =\bigoplus_{r+s=q} \Ext^{n+r}(E_{n-p}, A) \otimes \mathcal H^s(E_p^\vee )
\]
which is functorial in $A$ and converges to $\mathcal H^{p+q}(A)$.
\end{theorem}

The statement and proof of Theorem \ref{thm:Beilinson} can be found both in  \cite[Corollary 3.3.2]{RU}, in \cite[Section 2.7.3]{GO} and in \cite[Theorem 2.1.14]{BO}.


Let us assume next that the full exceptional collection  $\langle E_0, \ldots, E_n\rangle$ contains only pure objects of type $E_i=\mathcal E_i^*[-k_i]$ with $\mathcal E_i$ a vector bundle for each $i$, and moreover the right dual collection $\langle E_0^\vee, \ldots, E_n^\vee\rangle$ consists of coherent sheaves. Then the Beilinson spectral sequence is much simpler since
\[
E_1^{p,q}=\Ext^{n+q}(E_{n-p}, A) \otimes E_p^\vee=H^{n+q+k_{n-p}}(\mathcal E_{n-p}\otimes A)\otimes E_p^\vee.
\]

Note however that the grading in this spectral sequence applied for the projective space is slightly different from the grading of the usual Beilison spectral sequence, due to the existence of shifts by $n$ in the index $p,q$. Indeed, the $E_1$-terms of the usual spectral sequence are $H^q(A(p))\otimes \Omega^{-p}(-p)$ which are zero for positive $p$. To restore the order, one needs to change slightly the gradings of the spectral sequence from Theorem \ref{thm:Beilinson}. If we replace, in the expression
\[
E_1^{u,v} = \mathrm{Ext}^{v}(E_{-u},A) \otimes E_{n+u}^\vee=
H^{v+k_{-u}}(\mathcal E_{-u}\otimes A) \otimes \mathcal F_{-u}
\]
$u=-n+p$ and $v=n+q$ so that the fourth quadrant is mapped to the second quadrant, we obtain the following version of the Beilinson spectral sequence:

\begin{theorem}\label{use}
Let $X$ be a smooth projective variety with a full exceptional collection
$\langle E_0, \ldots, E_n\rangle$
where $E_i=\mathcal E_i^*[-k_i]$ with each $\mathcal E_i$ a vector bundle and $(k_0, \ldots, k_n)\in \zz^{\oplus n+1}$ such that there exists a sequence $\langle F_n=\mathcal F_n, \ldots, F_0=\mathcal F_0\rangle$ of vector bundles satisfying
\begin{equation}\label{order}
\mathrm{Ext}^k(E_i,F_j)=H^{k+k_i}( \mathcal E_i\otimes \mathcal F_j) =  \left\{
\begin{array}{cc}
\mathbb C & \textrm{\quad if $i=j=k$} \\
0 & \textrm{\quad otherwise}
\end{array}
\right.
\end{equation}
i.e. the collection $\langle F_n, \ldots, F_0\rangle$ labelled in the reverse order is the right dual collection of $\langle E_0, \ldots, E_n\rangle$.
Then for any coherent sheaf $A$ on $X$ there is a spectral sequence in the square $-n\leq p\leq 0$, $0\leq q\leq n$  with the $E_1$-term
\[
E_1^{p,q} = \mathrm{Ext}^{q}(E_{-p},A) \otimes F_{-p}=
H^{q+k_{-p}}(\mathcal E_{-p}\otimes A) \otimes \mathcal F_{-p}
\]
which is functorial in $A$ and converges to
\begin{equation}
E_{\infty}^{p,q}= \left\{
\begin{array}{cc}
A & \textrm{\quad if $p+q=0$} \\
0 & \textrm{\quad otherwise.}
\end{array}
\right.
\end{equation}
\end{theorem}

Let us focus on our case. Let us consider $F$ as an hyperplane section of $\pp^2\times\pp^2$ with the two natural projections $p_i:F\subset\pp^2\times\pp^2\arr\pp^2$ and the following rank two vector bundles:

\begin{equation}\label{pullbackcotangent}
\sG_1=p_1^*\Omega_{\mathbb P^2}^1(2h_1)\qquad \sG_2=p_2^*\Omega_{\mathbb P^2}^1(2h_2).
\end{equation}

\noindent Next, since it is easily computed that

$$
\dim(\Ext^1(\sO_F(-1,0),\sG_2(0,-2))=3,
$$

\noindent we denote by $\sG_3$ the rank $5$ vector bundle arising from the extension

\begin{equation}
\label{g3} 0\to\sG_2(0,-2)\to \sG_3\to\sO_F(-1,0)^{\oplus 3}\to 0,
\end{equation}
\noindent associated to a basis of the aforementioned group $\Ext^1(\sO_F(-1,0),\sG_2(0,-2))$.

Using the presentation of the flag variety as the projective bundle $F:=\pp(\Omega_{\pp^2}(1))\rightarrow \pp^2$, its derived category can be generated by the exceptional collection (see \cite{Orlov}):
$$
\{\sO_F(-1,-2), \sO_F(-1,-1), \sO_F(-1,0), \sO_F(0,-2), \sO_F(0,-1), \sO_F\}
$$
After two  left  mutations, we can obtain the exceptional collection
 \begin{equation}\label{col1} \{\sO_F(-1,-1)[-2], \sG_2(-1,-1)[-2],  \sG_1(-1,-1)[-1],
\sO_F(-1,0)[-1] ,  \sO_F(0,-1),  \sO_F\}
\end{equation}
and by another left mutation we obtain
$$ \{\sO_F(-1,-1)[-2], \sG_2(-1,-1)[-2],  \sG_1(-1,-1)[-1],
\sO_F(-1,0)[-1] , \sG_2(0,-2),  \sO_F(0,-1)\}.$$
Now we consider the pair $\sO_F(-1,0)[-1] , \sG_2(0,-2)$, from
$$
\mathbb L_{\sO_F(-1,0)[-1]} (\sG_2(0,-2)) 		\to	 \Ext^\bullet(\sO_F(-1,0)[-1] , \sG_2(0,-2)) \otimes \sO_F(-1,0)[-1]\to$$ $$\to \sG_2(0,-2) 	 		 \to	 \mathbb L_{\sO_F(-1,0)[-1]} (\sG_2(0,-2))[1]
$$ we get
$$0 \to \sG_2(0,-2) 	 		 \to	 \mathbb L_{\sO_F(-1,0)[-1]} (\sG_2(0,-2))[1]\to	 \Ext^1(\sO_F(-1,0) , \sG_2(0,-2)) \otimes \sO_F(-1,0)\to 0$$
and from (\ref{g3}) we obtain $$ \mathbb L_{\sO_F(-1,0)[-1]} (\sG_2(0,-2))=\sG_3[-1].$$
Hence we have the following exceptional collection
 \begin{equation}\label{col2} \{\sO_F(-1,-1)[-2], \sG_2(-1,-1)[-2],  \sG_1(-1,-1)[-1], \sG_3[-1],
\sO_F(-1,0)[-1] ,  \sO_F(0,-1)\}
\end{equation}
which will be the one considered in Theorem \ref{firstmonad}.

 \section{Monads on the flag variety $F=F(0,1,2)$}\label{sec:monad}
In this section we use the Beilinson Theorem in order to characterize the instanton bundles as the cohomology of two different monads, describe the relation between them and use it to give a presentation of the moduli space $MI_F(k)$ as a GIT-quotient.

\begin{theorem}\label{secondmonad}
Let $\sE$ be an instanton bundle with charge $k$ on $F$. Then, up to permutation, $\sE$ is the cohomology of a monad
\begin{equation}\label{mon2} 0 \to \sO_F(-1,0)^{\oplus k}\oplus \sO_F(0,-1)^{\oplus k} \xrightarrow{\alpha}  \sG_1(-1,0)^{\oplus k}\oplus \sG_2(0,-1)^{\oplus k}
\xrightarrow{\beta} \sO_F^{\oplus 2k-2} \to 0,\end{equation}
\noindent where $\sG_i$ is the pull-back of the twisted cotangent bundle $\Omega_{\pp^2}(2)$ from the two natural projections $p_i:F\subset\pp^2\times\pp^2\arr\pp^2$.

Reciprocally, the cohomology of such a monad defines a $k$-instanton.
\end{theorem}

\begin{proof}

We construct a Beilinson complex quasi-isomorphic to $\sE$ as in Theorem \ref{use} from (\ref{col1}) by calculating $H^i(\sE\otimes E_j)\otimes F_j$, with  $i,j \in \{0,1,2,3,4,5\}$ with

$$
E_5=\sO_F(-1,-1)[-2], E_4 = \sG_2(-1,-1)[-2], E_3 = \sG_1(-1,-1)[-1],
$$
$$
E_2 = \sO_F(-1,0)[-1] , E_1 = \sO_F(0,-1), E_0 = \sO_F
$$
and
$$
F_0 = \sO_F, F_1 = \sG_2(0,-1), F_2 = \sG_1(-1,0), F_3 =  \sO_F(0,-1), F_4 = \sO_F(-1,0), F_5 = \sO_F(-1,-1)
$$
Let us notice that the cohomological properties in (\ref{order}) are satisfied, therefore, our goal is to find all possible information on the following table

  \begin{center}\begin{tabular}{|c|c|c|c|c|c|c|c|c|c|c|}
\hline

$H^3$	&	$H^3$	&	$0$		&	$0$		&	$0$		&	$0$	\\
$H^2$	&	$H^2$	&	$H^3$	&	$H^3$	&	$0$		&	$0$	\\
$H^1$	& 	$H^1$	&	$H^2$	&	$H^2$	&	$H^3$	&	$H^3$	\\
$H^0$	& 	$H^0$	&	$H^1$	&	$H^1$	&	$H^2$	&	$H^2$	\\
$0$		&	$0$ 	 	&	$H^0$	&	$H^0$	& 	$H^1$	& 	$H^1$	\\
$0$		&	$0$		&	$0$		&	$0$		&	$H^0$	& 	$H^0$ \\
\hline
$\sE(-1,-1)$		& $\sE\otimes\sG_2(-1,-1)$	 & $\sE \otimes \sG_1(-1,-1)$	 & $\sE(-1,0)$		& $\sE(0,-1)$		&$\sE$\\
\hline
\end{tabular}
\end{center}

From the cohomological hypothesis it is straightforward to compute that $H^i(\sE(-1,-1))=0$, for $i=0,1,2,3$.\\
Considering the short exact sequences
$$
0 \longrightarrow \sE \otimes \sG_1(-1,-1) \longrightarrow \sE(0,-1)^3 \longrightarrow \sE(1,-1) \longrightarrow 0
$$
and
$$
0 \longrightarrow \sE(-2,-1) \longrightarrow \sE(-1,-1)^3 \longrightarrow \sE \otimes \sG_1(-1,-1) \longrightarrow 0
$$
we obtain directly
$$
H^0(\sE \otimes \sG_1(-1,-1)) = H^3(\sE \otimes \sG_1(-1,-1)) = 0,
$$
moreover, $H^2(\sE \otimes \sG_1(-1,-1)) \simeq H^3(\sE(-2,-1)) \simeq H^0(\sE(0,-1)) = 0$, hence we only need to compute
$$
H^1(\sE \otimes \sG_1(-1,-1)) \simeq H^2(\sE(-2,-1)) \simeq H^1(\sE(0,-1)).
$$
Using (\ref{eulerchar}) to compute the Euler characteristic, we get
$$
-h^1(\sE(0,-1)) = \chi(\sE(0,-1)) = 2 + \frac{1}{6}(6k) + \frac{1}{2}(2-4k)+\frac{1}{12}(-36) = -k.
$$
In the same way, we obtain that
$$
h^i(\sE \otimes \sG_2(-1,-1)) = \left\{
\begin{array}{cl}
k & \mbox{if}\:\:i=1,\\
0 & \mbox{if}\:\:i=0,2,3.
\end{array}
\right.
$$
Using the previous computation, we also get that
$$
h^i(\sE(0,-1)) = h^i(\sE(-1,0)) = \left\{
\begin{array}{cl}
k & \mbox{if}\:\:i=1,\\
0 & \mbox{if}\:\:i=0,2,3.
\end{array}
\right.
$$
Considering the short exact sequence
$$
0 \longrightarrow \sE(-2,-2) \longrightarrow \sE \otimes \sG_1(-3,-1) \longrightarrow \sE(-3,0) \longrightarrow 0
$$
and
$$
0 \longrightarrow \sE \otimes \sG_1(-3,-1) \longrightarrow \sE(-2,-1)^3 \longrightarrow \sE(-1,-1) \longrightarrow 0
$$
we can calculate $H^1(\sE\otimes \sG_1(-3,-1)) \simeq H^1(\sE(-2,-2)) \simeq H^2(\sE) = 0$. Moreover
$$
-h^1(\sE) = \chi(\sE) = 2-2k.
$$
Finally, the vanishing $H^0(\sE)=0$ implies $H^3(\sE)=0$. \\
Due to all previous computations, the cohomology table becomes
\begin{center}\begin{tabular}{|c|c|c|c|c|c|c|c|c|c|c|}
\hline

$0$	&	$0$	&	$0$	&	$0$	&	$0$	&	$0$	\\
$0$	&	$0$	&	$0$	&	$0$	&	$0$	&	$0$	\\
$0$	& 	$k$	&	$0$	&	$0$	&	$0$	&	$0$	\\
$0$	& 	$0$	&	$k$	&	$k$	&	$0$	&	$0$	\\
$0$	&	$0$ 	&	$0$	&	$0$	& 	$k$	& 	$2k-2$	\\
$0$	&	$0$	&	$0$	&	$0$	&	$0$	& 	$0$ \\
\hline
$\sE(-1,-1)$	&	 $\sE\otimes\sG_2(-1,-1)$	 & $\sE \otimes \sG_1(-1,-1)$	& $\sE(-1,0)$		& $\sE(0,-1)$		&$\sE$\\
\hline
\end{tabular}
\end{center}

From this table, since $Ext^k(F_i,F_j)=0$ for $k>0$ and any $i,j$, we get the claimed monad.

Reciprocally, the cohomology of a monad of type (\ref{mon2})  is an instanton bundle. Indeed  $H^0(\sG_1(-1,0)) = H^0(\sG_2(0,-1)) =0$,  combined with $H^1(\sO_F(-1,0))=H^1(\sO_F(0,-1))=0$, gives us that $H^0(\sE)=0$ and  $H^1(\sG_1(-2,-1)) = H^1(\sG_2(-1,-2)) = H^0(\sO_F(-1,-1))=0$,  combined with $H^2(\sO_F(-2,-1))=H^2(\sO_F(-1,-2))=0$, gives us that $H^1(\sE(-1,-1))=0$. Moreover, for each pair $(p,q)$ such that $p+q<0$, we get $H^0(\sG_1(p-1,q))=H^0(\sG_2(p,q-1))=H^1(\sO_F(p-1,q))=H^1(\sO_F(p,q-1))=0$, which implies $H^0(\sE(p,q))=0$, hence the $\mu$-semistability of the bundle $\sE$.
\end{proof}

Let us remark that the monad above is the analog of the monad for instanton bundles on $\mathbb P^3$ (see for instance \cite{AO} display $(1.1)$)
$$
0 \to \sO_{\pp^3}(-1)^{\oplus k} \xrightarrow{\alpha}  \Omega_{\pp^3}(1)^{\oplus k}
\xrightarrow{\beta} \sO_{\pp^3}^{\oplus 2k-2} \to 0.
$$

We are going to present now the second characterization of instanton bundles through a monad. This presentation was already proposed by Buchdahl (see \cite{Bu}, display $(3.9)$ with $n=2$ and $l=0$) and Donaldson (for charge $k=1$; see \cite{Don}, display $(14)$). Nevertheless, for the sake of completeness, we will go through the full computation in the next Theorem. After that, we will explicit the relation between the two monads. Later on, we will use the second characterization to describe the moduli space $MI_F(k)$ of $k$-instantons and its open subset $MI_F^s(k)$ of stable $k$-instantons.

\begin{theorem}\label{firstmonad}
Let $\sE$ be an instanton bundle with charge $k$ on $F$. Then, up to permutation, $\sE$ is the cohomology of a monad \begin{equation}\label{mon}0 \to \sO_F(0,-1)^{\oplus k}\oplus \sO_F(-1,0)^{\oplus k} \xrightarrow{\alpha}  \sO_F^{\oplus 4k+2}
\xrightarrow{\beta} \sO_F(1,0)^{\oplus k}\oplus \sO_F(0,1)^{\oplus k}\to 0.\end{equation}
Moreover, the monad obtained is self-dual, i.e. it is possible to find a non degenerate symplectic form $q: W \rightarrow W^*$, with $W$ a $(4k+2)$-dimensional vector space describing the copies of the trivial bundle in the monad, such that $\beta = \alpha^\lor \circ (q \otimes id_{\sO_F})$.

Reciprocally, any vector bundle with no global sections defined as the cohomology of such a monad is a $k$-instanton bundle.
\end{theorem}
\begin{proof}
We  construct a Beilinson complex quasi-isomorphic to $\sE$,
  by calculating $H^i(\sE\otimes E_j)\otimes F_j$ as in Theorem \ref{use} from (\ref{col2}), with  $i,j \in \{0,1,2,3,4,5\}$, where $$E_5=\sO_F(-1,-1)[-2]	, E_4=\sG_2(-1,-1)[-2], E_3= \sG_1(-1,-1) [-1],$$ $$E_2=\sG_3 [-1], E_1=\sO_F(0,-1)[-1], E_0=\sO_F(-1,0),$$ and $$F_0=\sO_F(1,0), F_1=\sO_F(0,1), F_2=\sO_F, F_3=\sO_F(0,-1), F_4=\sO_F(-1,0), F_5=\sO_F(-1,-1).$$
We have computed, from the previous result, all the entries of the cohomology table except one column, hence we already have the following:

  \begin{center}\begin{tabular}{|c|c|c|c|c|c|c|c|c|c|c|}
\hline
$0$			& $0$	 	& $0$		 	& $0$		 	& $0$			& $0$	\\
$0$		&$ 0$	 	& $0$	 	& $H^3$		& $0$			& $0$	\\
		$0$		& $k$		& $0$		& $H^2$		& $0$		& $0$	\\
	$0$ & $0$	& $k$	&	 $H^1$		& $k$		& $0$	\\
$0$			& $0$  			& $0$		& $H^0$		& $0$		& $k$	\\
$0$			& $0$			& $0$			& $0$		& $0$		& $0$ \\
\hline
$\sE(-1,-1)$		& $\sE\otimes\sG_2(-1,-1)$	 & $\sE \otimes \sG_1(-1,-1)$	 & $\sE\otimes\sG_3$		& $\sE(0,-1)$		&$\sE(-1,0)$\\
\hline
\end{tabular}
\end{center}

 Since  $$\chi(\sE\otimes\sG_2(0,-2))=\chi(\sE(0,-1)^{\oplus 3})-\chi(\sE)=-3k+2k-2=-2-k$$  and by the sequence ($\ref{g3}$) tensored by $\sE$ we have $$\chi(\sE\otimes\sG_3)=\chi(\sE(-1,0)^{\oplus 3})+\chi(\sE\otimes\sG_2(0,-2))=-3k-2-k=-4k-2.$$ This means that $H^i(\sE\otimes\sG_3)=0$ for $i\not= 1$ and we obtain

 \begin{center}\begin{tabular}{|c|c|c|c|c|c|c|c|c|c|c|}
\hline
$0$			& $0$	 	& $0$		 	& $0$		 	& $0$			& $0$	\\
$0$	&	$ 0$	 	& $0$	 	& $0$		& $0$			& $0$	\\
		$0$		& $k$		& $0$		& $0$		& $0$		& $0$	\\
	$0$&  $0$&	 $k$	&	 $4k+2$		& $k$		& $0$	\\
$0$			& $0$  			& $0$		& $0$		& $0$		& $k$	\\
$0$			& $0$			& $0$			& $0$		& $0$	& $0$ \\
\hline
$\sE(-1,-1)$	&	 $\sE\otimes\sG_2(-1,-1)$	 & $\sE \otimes \sG_1(-1,-1)$	 & $\sE\otimes\sG_3$		& $\sE(0,-1)$		&$\sE(-1,0)$\\
\hline
\end{tabular}
\end{center}

From this table, since $Ext^k(F_i,F_j)=0$ for $k>0$ and any $i,j$, we get the claimed monad.

From previous computations in cohomology, we know that $H^1(\sE \otimes \sG_1(-1,-1)) \simeq H^1(\sE(0,-1))^*$ and analogously $H^1(\sE \otimes \sG_2(-1,-1)) \simeq H^1(\sE(-1,0))^*$. Denote by $I_1$ the vector space of the first isomorphism and by $I_2$ the vector space of the second one, and denote by $W$ the cohomology group $H^1(\sE\otimes\sG_3)$.\\
The vector bundles in the obtained monad satisfy the required cohomological conditions in order to have a bijection between homomorphism of monads and the induced homomorphism in the cohomology bundles (see for example Lemma 4.1.3 in \cite{OSS}). Recalling that $\sE$ carries a unique symplectic structure, a consequence of the cited result gives us two isomorphisms
$$
\bar{q} : W \otimes \sO_F \rightarrow W^* \otimes \sO_F
$$
and
$$
h : (I_1 \otimes \sO_F(1,0)) \oplus (I_2 \otimes \sO_F(0,1)) \rightarrow (I_1 \otimes \sO_F(1,0)) \oplus (I_2 \otimes \sO_F(0,1))
$$
such that
$$
q^\lor = - q \mbox{ and } h \circ \beta = \alpha^\lor \circ q.
$$
This implies that the monad (\ref{mon}) is self-dual, therefore we have that $\beta = \alpha^\lor \circ (q \otimes id_{\sO_F})$.

Reciprocally, the cohomology of a monad of type (\ref{mon}) with $H^0(\sE)=0$  is an instanton bundle. Indeed  $H^0(\sO_F(0,-1)) = H^0(\sO_F(0,-1)) = H^1(\sO_F(-1,-1))=0$, combined with $H^2(\sO_F(-2,-1))=H^2(\sO_F(-1,-2))=0$, gives us that $H^1(\sE(-1,-1))=0$. Moreover, for each pair $(p,q)$ such that $p+q<0$, we get $H^0(\sO_F(p,q))=H^1(\sO_F(p-1,q))=H^1(\sO_F(p,q-1))=0$, which implies $H^0(\sE(p,q))=0$, hence the $\mu$-semistability of the bundle $\sE$.
\end{proof}

\begin{remark}
\begin{itemize}
\item[i)] In the case of instanton bundles on $\pp^{2n+1}$, the cohomology of the monad analogous to the one given at Theorem \ref{firstmonad} had automatically no non-zero sections (cf. \cite[Theo. 2.8]{AO2}). However, in the case of the flag variety we are dealing with, we need to add this extra condition to recover an instanton bundle, as it had been pointed out at Remark \ref{remarkglobalsections}.
\item[ii)] If $k=1$ the instanton bundle $\sE$  twisted by $\sO_F(1,1)$ is an  Ulrich bundle with $c_1(\sE(1,1))=(2,2)$ and $c_2(\sE(1,1))=4h_1h_2$ which is $\mu$-stable unless $\sE$ arises (up to permutation) from an extension

$$0\to \sO_F(1,-1)\to \sE\to\sO_F(-1,1)\to 0,$$
see \cite{CFM3}. Moreover since $\sE(1,1)$ is Ulrich on a Del Pezzo Threefold, hence, by \cite{CFM2}, we get $$\Ext^2(\sE,\sE)=0.$$
\end{itemize}
\end{remark}

\begin{remark}
As in the case of instanton bundles on the projective space, the two monads defining instanton bundles on the flag variety are closely related. Indeed, let us describe the monad defined in Theorem \ref{firstmonad} by the short exact sequences (the first and second display of the monad)
\begin{equation}\label{displaymonad}
\begin{array}{c}
0 \longrightarrow \sK \longrightarrow \sO^{4k+2} \longrightarrow \sO(1,0)^{\oplus k}\oplus \sO(0,1)^{\oplus k} \longrightarrow 0 \vspace{0.5 cm}\\
0 \longrightarrow \sO(-1,0)^k \oplus \sO(0,-1)^k \longrightarrow \sK \longrightarrow \sE \longrightarrow 0
\end{array}
\end{equation}
The first sequence fits in the following commutative diagram
$$
\xymatrix{
& 0 \ar[d] & 0  \ar[d] & 0 \ar[d] \\
0 \ar[r] & \sK \ar[d] \ar[r] & \sO^{\oplus 4k+2} \ar[d] \ar[r]^(.3){\beta} &  \sO(1,0)^{\oplus k}\oplus \sO(0,1)^{\oplus k} \ar[r] \ar[d] & 0\\
0 \ar[r] & \sG_1(-1,0)^{\oplus k} \oplus \sG_2(0,-1)^{\oplus k} \ar[d] \ar[r] & \sO^{\oplus 6k} \ar[d] \ar[r]^(.3){\beta'} &  \sO(1,0)^{\oplus k}\oplus \sO(0,1)^{\oplus k} \ar[r] \ar[d] & 0\\
0 \ar[r] & \sO^{\oplus 2k-2} \ar[d] \ar[r] & \sO^{\oplus 2k-2} \ar[d] \ar[r] & 0 \\
& 0 & 0
}
$$
indeed, we can complete the matrix representing $\beta$ with columns of the required forms of degree 1 in either the $x_i$'s or $y_i$'s, in order to keep $\ker \beta'$ with no global sections. Starting from this diagram, it is possible to induce the following one
$$
\xymatrix{
& 0 \ar[d] & 0 \ar[d] & 0 \ar[d] \\
0 \ar[r] & \sO(-1,0)^{\oplus k} \oplus \sO(0,-1)^{\oplus k} \ar[r] \ar[d] & \sK \ar[r] \ar[d] & \sE \ar[r] \ar[d] & 0 \\
0 \ar[r] & \sO(-1,0)^{\oplus k} \oplus \sO(0,-1)^{\oplus k} \ar[r] \ar[d] & \sG_1(-1,0)^{\oplus k} \oplus \sG_2(0,-1)^{\oplus k} \ar[r] \ar[d] & \sC \ar[r] \ar[d] & 0 \\
 & 0\ar[r] & \sO^{\oplus 2k-2}\ar[r]\ar[d] & \sO^{\oplus 2k-2} \ar[r] \ar[d]& 0\\
 & & 0 & 0
}
$$
which can be seen as the display of the monad defined in Theorem \ref{secondmonad}.
\end{remark}

The previous presentation allows us to prove the following result (compare with \cite[Section 2]{CO}).

\begin{theorem}
  The moduli space $MI_F(k)$ is the geometric GIT quotient $\fD_k^0/G_k$.
\end{theorem}
\begin{proof}
  We just saw that isomorphic classes of instanton bundles are in one-to-one correspondence with $G_k$-orbits in $\fD_k^0$. Moreover, their isotropy group is $\Lambda:=\pm(id_{Sp(W,J)},id_{GL(I_1)},id_{GL(I_2)}))$. Therefore, $G_k/\Lambda$ acts freely on $\fD_k^0$ and in particular all orbits are closed:$\fD_k^0\rightarrow\fD_k^0/(G_k/\Lambda)$ is a geometric quotient.
\end{proof}

In general, it is a challenging question to know whether a given moduli space is an affine variety. For the case of instanton bundles on $\pp^{2n+1}$ a positive answer was given in \cite{CO}.

Following \cite{CO}, we are going to call a matrix $A^t\in Hom((U\otimes I_1^*)\oplus (U^*\otimes I_2^*),W)$ \emph{degenerate} if there exists $u_1\in U$, $u_2\in U^*$, such that $u_2(u_1)=0$ and such that the induced linear map $A^t((-\otimes u_1)\oplus(-\otimes u_2)):I^*_1\oplus I^*_2\rightarrow W$ is non-injective. It is straighforward to show that $A\in W^*\otimes((I_1\otimes U)\oplus(I_2\otimes U^*))$ defines a surjective map if and only if $A^t$ is nondegenerate. Therefore we can conclude that $\fD_k^0=\{A\in \fD_k\mid A \text{ is injective and $A^t$ is nondegenerate}\}$. With these ingredients we are able to prove:

\begin{proposition}\label{affine-first}
$MI_F(1)$ is an affine variety.
\end{proposition}
\begin{proof}
Notice that in this case, the map of vector bundles $\tilde{A}:W\otimes\sO_F\rightarrow (I_1\otimes\sO_F(1,0))\oplus(I_2\otimes\sO_F(0,1))$ becomes, at the level of global sections, a square $(6\times 6)$-matrix associated to the map $A:W\rightarrow U\oplus U^*$. Let us denote by $D(A)$ the usual determinant of this map. Notice that $D$ is $G_k\times GL(U)$-invariant. Let us show that, for $A\in\fD_1$, $A\in\fD_1^0$ if and only if $D(A)\neq 0$.

\noindent First, obviously if $A$ is degenerate, then $D(A)=0$. Reciprocally, if $A\in\fD_1^0$ is defining an instanton bundle $\sE$ with associated first display
  $$
  0\rightarrow \sO_F(-1,0)\oplus\sO_F(0,-1)\rightarrow \sE\rightarrow \sK\rightarrow 0,
  $$
we obtain that $H^0(\sK)=H^0(\sE)=0$ and therefore, taking global sections at the second display
$$
  0\rightarrow H^0(\sK)=0 \rightarrow H^0(\sO_F^{6})\cong\mathbb{C}^6 \stackrel{A}{\rightarrow} H^0(\sO_F(1,0)\oplus\sO_F(0,1))\cong\mathbb{C}^6\rightarrow 0,
  $$
\noindent we obtain $A$ injective and $D(A)\neq 0$.

Therefore, we obtain that $\fD_1^0=\{A\in\fD_1\mid D(A)\neq 0\}$ is an affine variety $Spec(S)$. Therefore, by the famous Theorem of Hilbert and Nagata, $MI_F(1)=\fD_1^0/G_k=Spec(S^{G_k})$ is also affine.
\end{proof}

\begin{remark}
  \begin{itemize}
    \item[i)] In Theorem \ref{affine-two} we will obtain an stronger result concerning $MI_F(1)$. However, it seems that the approach used there can not be extended to deal with instanton bundles of higher charge.
    \item[ii)] On the other hand, the approach from Proposition \ref{affine-first}, namely, the correct identification of an invariant form $D$ on the vector space under consideration, is the one used in \cite[Theorem 1.2]{CO} (resp. \cite[Theorem C]{Fa2}) to prove that the moduli space of instantons of arbitrary charge on $\pp^3$ (resp. the three-dimensional quadric) is an affine variety. We state therefore the following conjecture.
  \end{itemize}
\end{remark}

\begin{conj}
$MI_F(k)$ is affine for any $k$.
\end{conj}

\section{Construction of instantons}
In this section we will construct, through an induction process, stable $k$-instanton bundles on the flag variety for each charge $k$. More concretely, we are going to prove the following

\begin{theorem}\label{existenceinst}
  Let $F\subset\pp^7$ be the flag variety. The moduli space of stable $k$-instanton bundles $MI_F^s(k)$ has a generically smooth irreducible component of dimension $8k-3$.
\end{theorem}

We will start describing completely the case of charge $1$  that is the base case for the induction. We give two approaches to this case: first we give a complete explicit characterization in terms of the monad defining them. Next, we give a more abstract construction using Serre correspondence. Finally, in the subsection \ref{sec-ind}, we take care of the induction step.

\subsection{Base case of induction.}
Let us consider the case of charge 1, i.e. a vector bundle which is cohomology of the monad
$$
\sO_F(-1,0) \oplus \sO_F(0,-1) \stackrel{A}{\longrightarrow} \sO_F^6  \stackrel{B}{\longrightarrow} \sO_F(1,0) \oplus \sO_F(0,1)
$$
The matrix $B$, up to a change of coordinates and action of the linear groups involved in the monad, is of the form
$$
\left[
\begin{array}{cccccccccccc}
x_0 & x_1 & x_2 & 0 & 0 & 0 \\
0 & 0 & 0 & y_0 & y_1 & y_2
\end{array}
\right]
$$
where we can use the coordinates $(x_0:x_1:x_2)\times(y_0:y_1:y_2)$ of the product $\mathbb{P}^2\times \mathbb{P}^2$. Notice that we can think the entries of the matrix in such coordinates, we just need to consider that they satisfy the equation defining the flag variety, which we can suppose to be of the form $x_0y_0 + x_1y_1 + x_2y_2=0$.

This means that $\ker B \simeq \sG_1(-1,0) \oplus \sG_2(0,-1)$ and the $1$-instanton bundle $\sE$ is given as the cokernel

$$
0 \longrightarrow\sO_F(-1,0) \oplus \sO_F(0,-1) \stackrel{A}{\longrightarrow} \sG_1(-1,0) \oplus \sG_2(0,-1)  \longrightarrow \sE\longrightarrow 0
$$
\noindent where the matrix $A$ is of the form
$$
A =
\left[
\begin{array}{cccccc}
f_1 & \gamma y_0\\
f_2 & \gamma y_1\\
f_3 & \gamma y_2\\
\delta x_0 & g_1\\
\delta x_1 & g_2\\
\delta x_2 & g_3
\end{array}
\right]
$$
where $(f_1\:\:\:f_2\:\:\:f_3)^t$ and $(g_1\:\:\:g_2\:\:\:g_3)^t$ are syzygies of, respectively, $(x_0\:\:\:x_1\:\:\:x_2)$ and $(y_0\:\:\:y_1\:\:\:y_2)$, and $\gamma,\delta \in \mathbb{C}$. Notice that it completely agrees with the fact that the family of charge 1 instanton is 5-dimensional; indeed, each of the two syzygies is 3-dimensional as vector space, and we still have to apply the action of the 2-dimensional linear group of the automorphisms of $\sO_F(-1,0) \oplus \sO_F(0,-1)$. In particular, $MI_F(1)$ is the open subset of

$$
\pp\Hom(\sO_F(-1,0) \oplus \sO_F(0,-1), \sG_1(-1,0) \oplus \sG_2(0,-1))\cong\pp(H^0(\sG_1)\oplus H^0(\sG_2))
$$
\noindent defined as the complement of the hypersurface given by the vanishing of the determinant of the following matrix,  where$f_i^{x_j}$ (resp. $g_i^{y_j}$) stands for the coefficient of the polynomial $f_i$ (resp. $g_i$) with respect the variable $x_j$ (resp $y_j$),

$$
A =
\left[
\begin{array}{ccccccccc}
f_1^{x_0} & f_1^{x_1} & f_1^{x_2} & \gamma   &0 & 0\\
f_2^{x_0} & f_2^{x_1} & f_2^{x_2} &   0      &\gamma   &0 \\
f_3^{x_0} & f_3^{x_1} & f_3^{x_2}   &0 & 0 & \gamma \\
\delta & 0 &0 & g_1^{y_0} & g_1^{y_1} & g_1^{y_2} \\
 0 &\delta & 0 & g_1^{y_0} & g_1^{y_1} & g_1^{y_2} \\
 0 &0 & \delta & g_3^{y_0} & g_3^{y_1} & g_3^{y_2}
\end{array}
\right]
$$
\noindent So we can state an improvement of Proposition \ref{affine-first}:
\begin{theorem}\label{affine-two}
The moduli space of $1$-instanton bundles $MI_F(1)$ is an affine irreducible smooth variety of dimension $5$.
\end{theorem}

The charge 1 instanton bundles $\sE$ restrict trivially on the generic conic and for the generic line for each of the two families. Moreover, we have $\Ext^2(\sE,\sE) = 0$, because they are Ulrich if stable. Therefore, we obtained the candidate to start the induction process described in the next section. Furthermore, the following result characterizes the semistable case.
\begin{proposition}
Let $\sE$ be a $1$-instanton on the flag variety $F$. Then the vector bundle $\sE$ is properly semistable if and only if at least one of two syzygies is zero. Moreover, it splits as $\sE\simeq \sO_F(-1,1) \oplus \sO_F(1,-1)$ if and only if both $f_i$'s and $g_i$'s are zero.
\end{proposition}
\begin{remark}
Vanishing either the $f_i$'s or the $g_i$'s corresponds with the choice of one of two possible families given by the extension, i.e.
\begin{equation}\label{semiext}
\begin{array}{c}
0 \longrightarrow \sO_F(1,-1) \longrightarrow \sE \longrightarrow \sO_F(-1,1) \longrightarrow 0\vspace{0.5cm}\\
0 \longrightarrow \sO_F(-1,1) \longrightarrow \sE \longrightarrow \sO_F(1,-1) \longrightarrow 0
\end{array}
\end{equation}
\end{remark}
\begin{proof}
Let us recall the short exact sequences
\begin{equation}\label{canonicaltang}
\begin{array}{cccc}
0 \longrightarrow \sO_F(-1,0) \longrightarrow \sG_2(0,-1)  \longrightarrow \sO_F(1,-1) \longrightarrow 0\vspace{0.5cm}\\
0 \longrightarrow \sO_F(0,-1) \longrightarrow  \sG_1(-1,0) \longrightarrow \sO_F(-1,1) \longrightarrow 0\\
\end{array}
\end{equation}
Also recall that we have already proven that $\sE$ is semistable if only if it is one of the extensions described in  (\ref{semiext}). Consider the first one; it occurs if and only $\sE$ fits in the following diagram
$$
\xymatrix{
& 0 \ar[d] & 0 \ar[d] & 0 \ar[d]\\
0 \ar[r] & \sO_F(-1,0) \ar[r] \ar[d] & \sG_2(0,-1) \ar[d] \ar[r] & \sO_F(1,-1) \ar[r] \ar[d] & 0 \\
0 \ar[r] & \sO_F(-1,0)\oplus \sO_F(0,-1) \ar[r] \ar[d] & \sG_1(-1,0) \oplus \sG_2(0,-1) \ar[r] \ar[d] & \sE \ar[r] \ar[d] & 0\\
0 \ar[r] & \sO_F(0,-1) \ar[r] \ar[d] & \sG_1(-1,0) \ar[r] \ar[d] & \sO_F(-1,1) \ar[r] \ar[d] & 0\\
& 0 & 0 & 0
}
$$
We have the previous commutative diagram if and only if the syzygy given by $(g_1 \:\:\: g_2 \:\:\: g_3)$ is equal to zero.

Analogously, the other extension appears if and only if the syzygy given by $(f_1 \:\:\: f_2 \:\:\: f_3)$ is equal to zero.

Finally, it comes directly from the previous diagram that $\sE$ splits as $\sO(1,-1) \oplus \sO(-1,1)$ if and only if both syzygies vanish.
\end{proof}

\begin{example}
We show a specific example of instanton of charge 2, which we will know to be stable because its second Chern class is not a perfect square (see Proposition \ref{gieseker}). Indeed, consider the monad
$$
\sO_F(-1,0)^{\oplus 2} \oplus \sO_F(0,-1)^{\oplus 2} \stackrel{A}{\longrightarrow} \sO_F^{\oplus 10} \stackrel{B}{\longrightarrow} \sO_F(1,0)^{\oplus 2} \oplus \sO_F(0,1)^{\oplus 2}
$$
with the matrices defining the maps are given by
$$
A =
\left[
\begin{array}{ccccccccccccccccccccccccccc}
-y_0 & 0 & -x_0-x_2 & 0\\
-y_1 & 0 & -x_2 & -x_0\\
0 & 0 & 0 & -x_2\\
y_0 & y_2 & 0 & -x_1\\
y_2 & y_0 & 0 & 0 \\
y_0+y_1 & y_2 & 0 & -x_1\\
0 & y_1 & 0 & 0 \\
y_1 & 0 & 0 & x_0\\
0 & 0 & x_0 & x_1\\
-y_2 & 0 & x_0+x_1 & 0
\end{array}
\right]
\:\:\:\:\:
B =
\left[
\begin{array}{ccccccccccccccccccccccccccc}
x_0 & x_1 & 0 & 0 & 0 & 0 & 0 & 0 & x_0 & x_2\\
0 & x_0 & x_1 & 0 & 0 & 0 & 0 & x_0 & x_2 & 0\\
0 & 0 & -y_2 & -y_1 & 0 & 0 & y_2 & y_0 & 0 & 0\\
0 & 0 & 0 & -y_2 & -y_1 & y_2 & y_0 & 0 & 0 & 0
\end{array}
\right]
$$
As in the previous case, we denote by $(x_0:x_1:x_2)\times(y_0:y_1:y_2)$ the coordinates of the product $\mathbb{P}^2\times \mathbb{P}^2$, keeping in mind that they satisfy the equation $x_0y_0 + x_1y_1 + x_2y_2=0$ of the flag variety.
\end{example}

\subsection{Alternative proof of the existence of $1$-instantons}

In this subsection, using Serre correspondence, we are going to show the existence of stable $1$-instanton bundles with trivial splitting type on the general line from both families.
The proof will be similar to the one given on \cite{Fa2}.

\begin{theorem}\label{existencechargeone}
  Let $F$ be the flag variety. Then there exists a family of dimension $5$ of stable $1$-instantons $\sE$ on $F$. Moreover, they have trivial splitting type on the generic conic and line (from both families). Indeed, $\sE(1,1)$ are Ulrich bundles on $F$.
\end{theorem}
\begin{proof}
  Let us consider $S$ a hyperplane section of $F$. $S$ is a del Pezzo surface of degree $6$ embedded by the anticanonical line bundle $H_{F_{\mid S}}:=-K_{S}$. Therefore, $S$ can be seen as the blow-up at three generic points of $\pp^ 2$. Using standard terminology, $-K_S=3h-e_1-e_2-e_3$ where $l$ denotes the pullback of the class of a line in $\pp^2$ and $e_i$ are the exceptional divisors. Let us consider a general curve $C$ of class $3l-e_1$. It is a smooth elliptic curve of degree $8$. using the short exact sequence that relates the normal bundles:
$$0\arr\sN_{C,S}\arr\sN_{C,F}\arr\sN_{{S,F}_{\mid C}}\arr 0$$
\noindent we obtain
$$0\arr\sO_C(C)\arr\sN_{C,F}\arr\sO_C(H_C)\arr 0.$$
\noindent By easy Riemann-Roch computations, we obtain $\hh^0(\sN_{C,F})=16$ and $\hh^1(\sN_{C,F})=0$. Deformations theory tells us that the Hilbert scheme $Hilb^{8t}(F)$ of curves on $F$ with Hilbert polynomial $p(t):=8t$ is smooth of dimension $16$ at the point $[C]$.

Take a general deformation $D$ of $C$. It will be a smooth non-degenerate elliptic curve with Chern class $[D]=4h_1h_2\in A^*(F)$. Now, a non-zero element of
$$
\Ext^1(\sI_{D\mid F}(2,2),\sO_F)\cong\Ext^1(\sI_{D\mid F},\omega_F)\cong\HH^1(\omega_D)\cong\cc
$$
\noindent provides, through Serre correspondence, an unique extension

\begin{equation}\label{serre}
  0\arr\sO_F\arr\sG\arr\sI_{D\mid F}(2,2)\arr 0.
\end{equation}

Let us show that $\sG$ is an Ulrich sheaf, using its characterization by the minimal $\sO_{\pp^7}$-resolution (see \cite[Prop. 2.1]{ESW}). In order to do this, we will use that, being $F$ and $D$ quasi-minimal varieties on $\pp^7$, we know their minimal $\sO_{\pp^7}$-resolution. Indeed,
$$
0\arr\sO_{\pp^7}(-6)\arr\sO_{\pp^7}(-4)^{9}\arr\sO_{\pp^7}(-3)^{16}\arr\sO_{\pp^7}(-2)^9\arr\sO_{\pp^7}\arr\sO_F\arr 0,
$$
\noindent and

$$
0\arr\sO_{\pp^7}(-8)\arr\sO_{\pp^7}(-6)^{20}\arr\sO_{\pp^7}(-5)^{64}\arr\sO_{\pp^7}(-4)^{90}\arr\sO_{\pp^7}(-3)^{64}\arr
$$
$$\arr\sO_{\pp^7}(-2)^{20}\arr\sO_{\pp^7}\arr\sO_D\arr 0.
$$

\noindent Applying the mapping cone Theorem to the short exact sequence (which is exact due to $F$ being aCM):
$$
0\arr\sI_{F\mid\pp^7}\arr\sI_{D\mid\pp^7}\arr\sI_{D\mid F}\arr 0
$$

\noindent we obtain the resolution of $\sI_{D\mid F}$:

$$
0\arr\sO_{\pp^7}(-8)\arr\sO_{\pp^7}(-6)^{21}\arr\sO_{\pp^7}(-5)^{64}\arr\sO_{\pp^7}(-4)^{81}\arr\sO_{\pp^7}(-3)^{48}\arr
$$
$$\arr\sO_{\pp^7}(-2)^{11}\arr\sI_{D\mid F}\arr 0.
$$
Finally, applying the horseshoe Lemma to (\ref{serre}), $G$ has a linear $\sO_{\pp^7}$-resolution

$$
0\arr\sO_{\pp^7}(-4)^{12}\arr\sO_{\pp^7}(-3)^{48}\arr\sO_{\pp^7}(-2)^{72}\arr\sO_{\pp^7}(-1)^{48}\arr
$$
$$\arr\sO_{\pp^7}^{12}\arr\sG\arr 0,
$$
\noindent namely $\sG$ is an Ulrich rank $2$ bundle on $F$ with Chern classes $c_1(\sG)=2h_1+2h_2$ and $c_2(\sG)=4h_1h_2$. Now, if we define $\sE:=\sG(-1,-1)$, $\sE$ is an $1$-instanton bundle on $F$. Let us now show that $\sE$ sits on a generically smooth component of $MI(1)$ of dimension $5$. Namely, we compute the dimensions $\ext^i(\sE,\sE) (=\hh^i(\sE\otimes\sE)$ thanks to  $\sE^{\vee}\cong\sE$). First of all, we know from Lemma \ref{simple} that $\sE$ is simple: $\ext^0(\sE,\sE)=1$. On the other hand, tensoring the exact sequence (\ref{serre}) twisted by $\sO_F(-1,-1)$ with $\sE$ and considering  the cohomology vanishings of an Ulrich bundle, we get $\hh^i(\sE\otimes\sE)=\hh^i(\sE\otimes\sI_{D\mid F}(1,1))$. Moreover, from

$$
0\arr\sE\otimes\sI_{D\mid F}(1,1)\arr\sE(1,1)\arr\sE\otimes\sO_D(1,1)\arr 0
$$

we get $\hh^2(\sE\otimes\sI_{D\mid F}(1,1))=\hh^{1}(\sE\otimes\sO_D(1,1))$. But $\sE\otimes\sO_{D}(1,1)\cong\sN_{D,F}$ and $D$, being a general deformation of $C$, verifies $\hh^1(\sN_{D,F})=0$ and $\hh^0(\sN_{D,F})=16$ from where we obtain the result.

Let us observe that, alternatively, the dimension of the component of $MI(1)$ where $\sE$ sits could have also been computed using that, through Serre correspondence, there is a bijection, at least locally, between pairs $(\sE(1,1), s)$ -where $\sE(1,1)$ is an Ulrich bundle and $s\in\pp(\HH^0(\sE(1,1)))$ is a non-zero global section- and pairs $(D,t)$ where $D\subset F$ is an elliptic normal curve and $0\neq t\in\pp(\Ext^1(\sI_{D\mid F}(2,2),\sO_F))$ . Therefore, since $\hh^0(\sE(1,1))=12$, $\Ext^1(\sI_{D\mid F}(2,2),\sO_F)\cong\cc$ and $Hilb^{8t}(F)$ has dimension $16$ at $D$, we get that the dimension of the considered component of $MI(1)$ is $5$.

 We know that Ulrich bundles are semistable (see \cite{CHGS}).  Let us now show that the generic Ulrich bundle constructed above is stable: otherwise, by Hoppe's criterion, it would fit in a short exact sequence of the form:
  $$
  0\longrightarrow \sO_F(l,-l)\longrightarrow\sE\longrightarrow\sO_F(-l,l)\longrightarrow 0
  $$
\noindent with $l\in\Z$. Comparing Chern classes, it turns out that $l\in\{1,-1\}$. But properly semistable bundles sitting on one of these two exact sequences form families of dimension $2$, from the easy calculation $\ext^1(\sO_F(1,-1),\sO_F(-1,1))=\hh^1(\sO_F(-2,2))=3$.

It only remains to prove the trivial splitting type with respect to lines and conics on $F$. Let us start with the conic case. For this, let us consider a generic conic $A$ with class $l-e_1$ on the del Pezzo surface $S$. It will cut $C$ on two points $\{x,y\}$. Tensoring the short exact sequence
$$
0\arr\sI_{C\mid F}(1,1)\arr\sO_F(1,1)\arr\sO_C(1,1)\arr 0
$$

\noindent by $\sO_{A}$ we see that $\sI_{C\mid F}(1,1)\otimes\sO_{A}\cong\sO_A\oplus\sO_{\{x,y\}}$. Tensoring again (\ref{serre}) by $\sO_A(-1,-1)$ we get a surjection $\sE_{|A}$ to $\sO_A\oplus\sO_{\{x,y\}}$. Therefore, the only possibility is $\sE_{|A}\cong\sO_A^2$. So, by semicontinuity, the same will be true for the $1$-instanton bundle associated to a generic deformation $D$ of $C$ inside $Hilb^{8t}(F)$.

For the case of a general line $L$ on any of the two families $\pp^2$, we should change slightly the argument: in this case, we start with a smooth elliptic curve $C$ of degree $8$ on the del Pezzo $S$ from the linear system $5l-3e_1-2e_2-2e_3$. Now, the line $L$ with class $l-e_2-e_3$ will cut $C$ on a single point $x$ and therefore $\sI_{C\mid F}(1,1)\otimes\sO_{L}\cong\sO_L\oplus\sO_{\{x\}}$. Now the previous argument tells us that $\sE_{|L}\cong\sO_L^2$. Notice that, fixed the line $L$ on $S$, we have freedom on the choice
of the presentation of $S$ as a blow-up to assure that $L$ has class $l-e_2-e_3$. Therefore, this argument covers both families of lines on $F$.

\end{proof}

\subsection{The induction step}\label{sec-ind}
We complete here, by induction, the proof of Theorem \ref{existenceinst}. So let us suppose that this theorem is true for instantons of charge $k$, namely we can take an instanton bundle $\sE$ on the flag variety with $c_2(\sE) = k h_1 h_2$ and $c_1(\sE)=0$, stable and satisfying $h^2(\sE \otimes \sE) = 0$, belonging to a family of dimension $8k-3$ . Recall that $\sE$ has trivial restriction on the generic conic, as shown in Proposition \ref{trivial-conic}.

Consider the following short exact sequence

\begin{equation}\label{indstep}
0 \longrightarrow \sF \longrightarrow \sE \longrightarrow \sO_{C}(1) \longrightarrow 0
\end{equation}
which gives us a torsion free sheaf $\sF$ with $c_2(\sF) = (k+1)h_1 h_2$, $c_1(\sF) = c_3(\sF) = 0$ and $H^1(\sF(-1,-1))=0$. Moreover for the generic conic $C$ in the flag variety, we have trivial restrictions, i.e. $\sF_{|C} \simeq \sE_{|C} \simeq \sO_C^{\oplus 2}$. Applying $\Hom(\sE, -)$ and $\Hom(-,\sF)$ to (\ref{indstep}), we obtain $\Ext^2(\sF,\sF)=0$. Indeed, applying the first functor we get
$$
0 = \Ext^1(\sE,\sO_C(1)) \simeq H^1(\sO_C(1)^2) \rightarrow \Ext^2(\sE,\sF) \rightarrow \Ext^2(\sE,\sE) = 0
$$
and applying the second functor
$$
\Ext^2(\sE,\sF) \rightarrow \Ext^2(\sF,\sF) \rightarrow \Ext^3(\sO_C(1),\sF).
$$
By Serre duality $\Ext^3(\sO_C(1),\sF) \simeq \Hom(\sF, \sO_C(1) \otimes \sO_F(-2,-2))$ and, if the last group is not zero, we would have a homomorphism when restricting to the conic. Consider therefore the restriction of (\ref{indstep}) to $C$, getting
$$
0 \rightarrow \mathcal{T}or(\sO_C, \sO_C(1)) \rightarrow \sF_{|C} \rightarrow \sO_C^2 \rightarrow \sO_C(1) \rightarrow 0
$$
and applying $\Hom(-,\sO_C(-3))$ we get
$$
0 = \Hom(\sO_C(-1),\sO_C(-3)) \rightarrow \Hom(\sF_{|C},\sO_C(-3)) \rightarrow \Hom(\mathcal{T}or(\sO_C, \sO_C(1)),\sO_C(-3)) =0
$$
Putting all together, we get that $\Ext^2(\sF,\sF)=0$. Moreover, $\sF$ is stable, being a subsheaf of $\sE$ with the same slope.\\
Our goal will be to prove that inside $M_F^s(2,0,(k+1)h_1h_2,0)$, the moduli space of stable rank 2 sheaves with those fixed Chern classes, $\sF$ can be deformed to a vector bundle.
In order to know the local dimension of $M_F^s(2,0,(k+1)h_1h_2,0)$ at the point $[\sF]$, we would like now to compute
$$
1- \ext^1(\sF,\sF) = \chi (\sF,\sF) = \chi(\sE,\sF) - \chi(\sO_{C}(1),\sF).
$$
Directly from (\ref{indstep}), we compute $\chi(\sE,\sF) = \chi(\sE,\sE) - \chi(\sE,\sO_C(1))= (4-8k) - 4 = -8k$ and $\chi(\sO_{C}(1),\sF) = \chi(\sO_{C}(1),\sE) - \chi(\sO_{C}(1),\sO_C(1)) = 4 - 0$,  using the formula in (\ref{dim-chi-end}), the restriction of $\sE$ on the conic and finally the resolution of the structure sheaf $\sO_C$. We can conclude that $\ext^1(\sF,\sF) = 8k+5$.\\

Let us denote by $\tilde{\sF}$ a general deformation of $\sF$. Being $\tilde{\sF}$ a torsion free sheaf, we have a canonical short exact sequence
$$
0\longrightarrow \tilde{\sF} \longrightarrow \tilde{\sF}^{**} \longrightarrow T \longrightarrow 0
$$
where $T$ is a sheaf supported on a scheme of dimension at most 1. All the required cohomological properties for the subsequent steps will still hold because of the semicontinuity of the dimension of the $\Ext$ groups when considering a deformation in the moduli space of stable rank 2 vector bundles with fixed Chern classes.

The stability of $\tilde{\sF}$ implies $h^0(\tilde{\sF}^{**}(-1,-1)) =0$. On the other hand, being $h^1(\tilde{\sF}(-1,-1))=0$ directly form the definition of $\sF$ and the cohomological instanton condition on $\sE$, we get that $h^0(T(-1,-1))=0$. This tells us that the support of $T$ does not have isolated points, i.e. it is of pure dimension 1 when non empty. Let us  observe that we can suppose that its support is not empty, otherwise $\tilde{\sF}$ would already be a $(k+1)$-instanton bundle, as wanted. Our next goal is to compute $c_2(T)$ to get the degree of such support. First of all, let us notice that $\tilde{\sF}^{**}$ is a 2-rank reflexive sheaf on the flag variety, therefore $c_3(\tilde{\sF}^{**}(\alpha_1,\alpha_2))$ is invariant for each twist by a line bundle $\sO_F(\alpha_1,\alpha_2)$. We will denote such Chern class by $c$. Recall that, by Lemma \ref{Chern-reflexive-flag}, we have $c\geq 0$. Moreover, also $c_2(T(\alpha_1,\alpha_2))$ is invariant by twist, indeed it is minus the class on the Chow ring of the support of the sheaf $T$.

Considering the relation
$$
c_3(T(\alpha_1,\alpha_2)) = c - 2(\alpha_1 h_1 + \alpha_2 h_2)c_2(T)
$$
and using the Hirzebruch-Riemann-Roch formula, we compute
$$
\chi(T(\alpha_1,\alpha_2)) = \frac{1}{2}c - c_2(T)\left((1+\alpha_1)h_1 + (1+\alpha_2)h_2\right).
$$
Choosing $\alpha_1 = \alpha_2 = -1$ we get
$$
-\frac{1}{2} c = -\chi(T(-1,-1)) =h^1(T(-1,-1)) \geq 0,
$$
hence $c=0$, which means that $\tilde{\sF}^{**}$ is locally free.
Denoting $c_2(T) = \beta_1 h_1^2 + \beta_2 h_2^2$ and recalling that $\beta_1 + \beta_2<0$, being related to minus the degree of the support of $T$, we can suppose, without loss of generality, that $\beta_2<0$. Taking a pair of negative $\alpha_1,\alpha_2$ with $\alpha_1+\alpha_2\ll 0$ and $\alpha_1\ll \alpha_2$, we obtain the following cohomological relations
$$
-\chi(T(\alpha_1,\alpha_2)) = h^1(T(\alpha_1,\alpha_2)) \leq h^2(\tilde{\sF}(\alpha_1,\alpha_2)) \leq h_1(\sO_{C}(\alpha_1+\alpha_2+1))
$$
where the inequality on the middle holds because the sheaf $\tilde{\sF}^{**}$ is locally free and the right equality follows from the semicontinuity of the dimension of the $\Ext$ groups. We thus obtain
$$
(1 + \alpha_1)\beta_2 + (1 + \alpha_2)\beta_1 \leq - \alpha_1 - \alpha_2 - 2
$$
that gives us $\beta_2 =  -1$. Therefore $\beta_1 <1$, and we can exclude all cases $\beta_1 \leq -2$, taking $\alpha_2 \ll 0$, hence $\beta_1=-1,0$.

This means that if $\tilde{\sF}$ is not locally free, it would fit in one of the following short exact sequences: if $\beta_1=\beta_2 = -1$ we have
\begin{equation}\label{deform-conic}
0\longrightarrow \tilde{\sF} \longrightarrow \tilde{\sF}^{**} \longrightarrow \sO_{C}(1) \longrightarrow 0
\end{equation}
or
\begin{equation}\label{deform-degconic}
0\longrightarrow \tilde{\sF} \longrightarrow \tilde{\sF}^{**} \longrightarrow \sO_{L_1} \oplus \sO_{L_2} \longrightarrow 0;
\end{equation}
if $\beta_1 =0$ and $\beta_2 =-1$ (resp., if $\beta_1 =-1$ and $\beta_2 =0$) we have
\begin{equation}\label{deform-line2}
0\longrightarrow \tilde{\sF} \longrightarrow \tilde{\sF}^{**} \longrightarrow \sO_{L_2} \longrightarrow 0,
\end{equation}
or, respectively
\begin{equation}\label{deform-line1}
0\longrightarrow \tilde{\sF} \longrightarrow \tilde{\sF}^{**} \longrightarrow \sO_{L_1} \longrightarrow 0.
\end{equation}

We have shown that, if the deformation $\tilde{\sF}$ is not locally free, than it always fits into one of the four exact sequences described in (\ref{deform-conic}), (\ref{deform-degconic}), (\ref{deform-line2}) and (\ref{deform-line1}).\\
Consider the first case. The family defined by such a short exact sequences has dimension $8k+4$. Indeed, this can be computed observing that the dimension of the moduli of $\tilde{\sF}^{**}$, which is a $k$-instanton, is equal $8k-3$. Then we have to add the choice of a conic $C$ (from a space of dimension $4$) plus a surjective morphism $\sE \rightarrow \sO_C(1)$ (i.e, an element of the space $\pp(H^0(\sO_C(1)^2))$, of dimension $3$).\\
Consider the second case. The family defined by such short exact sequences has dimension $8k+1$. Indeed, this can be computed observing that, here again, the dimension of the moduli of vector bundles $\tilde{\sF}^{**}$, which is a $k$-instanton, is equal to $8k-3$, $4$ given by the choice of the two lines $L_1$ and $L_2$, respectively of the first and second family in the flag variety, while a surjective morphism $\sE \rightarrow \sO_{L_1} \oplus \sO_{L_2}$ is still uniquely defined through the restriction $\sE_C \simeq \sO_C^2$.\\
Consider the third case. The family defined by such short exact sequence has dimension $8k+4$. Indeed, this can be computed observing that the dimension of the moduli of vector bundles $\tilde{\sF}^{**}$, with  $c_1(\tilde{\sF}^{**})=0$ and $c_2(\tilde{\sF}^{**})=kh_1^2 + (k+1)h_2^2$ has dimension $8k+1$, $2$ are the lines $L_2$ of the second family and $1$ is the dimension of $\mathbb{P}(H^0(\sE_{|L_2}))$, which represents the choice of one of the two summands of the trivial restriction in order to construct the surjection $\sE \rightarrow \sO_{L_2} $ .\\
The fourth case is completely analogous to the third one, and here again the family defined by such short exact sequence has dimension $8k+4$.\\
For any of the described cases, which we recall are all the possible ones for a non locally free deformation, the defined families do not fill all the possible deformations of the original sheaf $\sF$. We can conclude that we can always deform $\sF$ to a stable locally free sheaf $\tilde{\sF}$, with $c_1(\sF) =0$, $c_2(\sF)=(k+1)h_1h_2$ and $H^1(\tilde{\sF}(-1,-1))$, hence by definition it is a stable $(k+1)$-instanton bundle. It also has trivial splitting type on the generic conic. Moreover, we have constructed such $\tilde{\sF}$ in order to have $\Ext^2(\tilde{\sF},\tilde{\sF})=0$, therefore satisfies all our hypothesis to apply induction.

\section{Jumping rational curves}
Now we are going to define the notion of jumping conic. For instanton bundles on $\pp^3$, jumping lines have been thoroughly studied. Here we propose an analogous definition for conics on the flag variety that deals at once with the irreducible and reducible case:

\begin{definition}
  Let $\sE$ be an instanton bundle on the flag variety $F$. A conic $C\subset F$ (irreducible or not) is a jumping conic of type $(a,b)$ if it satisfies $H^1(\sE_{|C}(-1,0))=a$ and $H^1(\sE_{|C}(0,-1))=b$. $C$ is said to have trivial splitting  type when it has type $(0,0)$.
\end{definition}

Let us give some insight to our definition. Indeed, suppose first that $C\subset F$ is an irreducible conic, $C\cong\pp^1$. In that case, $\sO_F(-1,0)_{|C}=\sO_F(0,-1)_{|C}=\sO_C(-1)$ and for an instanton bundle $\sE$ we have $\sE_{|C}\cong\sO_C(-a)\oplus\sO_C(a)$ if and only if $H^1(\sE_{|C}(-1,0))=H^1(\sE_{|C}(0,-1))=a$ if and only if it is a jumping conic of type $(a,a)$.

On the other hand, for a reducible conic $C=L_1\cup L_2$ for lines $L_i$ intersecting transversely on a single point. In this case, it is well-known that $Pic(C)\cong \mathbb{Z}^2$, where the isomorphism is given by $\sL\to(deg_{L_1}(\sL),deg_{L_2}(\sL))$. Therefore, for an instanton $\sE$ on $F$ the restriction to $C$ is of the form $\sE_C\cong\sO_C(a,b)\oplus\sO_C(-a,-b)$  if and only if it is a jumping conic of type $(a,b)$.

\noindent Let us consider the exact sequence associated to a conic $C$
$$0\to\sO_F(-1,-1)\to\sO_F(0,-1)\oplus \sO_F(-1,0) \to  \sO_F\to\sO_C\to 0$$
Writing the above
sequence in families with respect to global sections of $\sO_F(0,-1)\oplus \sO_F(-1,0) $, we get the description
of the universal conic $\mathcal C\subset F\times H$

\begin{equation}\label{uni}0 \to\sO_F(-1,-1)\boxtimes\sO_H(-1,-1)\to\sO_F(0,-1)\boxtimes\sO_H(0,-1)\oplus \sO_F(-1,0)\boxtimes\sO_H(-1,0) \to  \sO_{F\times H}\to\sO_{\mathcal C}\to  0.\end{equation}
We denote by $\mathcal D_\sE$ the locus of jumping conics of
an instanton $\sE$, and by $i$
its embedding in $H$.

\begin{proposition}\label{jumpingdivisor}
Let $\sE$ be a $k$-instanton on $F$. Then $\mathcal D_{\sE}$ is a divisor of type $(k,k)$  equipped with a sheaf
$G$
fitting into
\begin{equation}\label{jump}0 \to\sO_H(-1,-1)^{\oplus k}\oplus \sO_H(-1,0)^{\oplus k}\to  \sO_H^{\oplus k}\oplus \sO_H(-1,0)^{\oplus k}\to i_*G\to  0.
\end{equation}

\end{proposition}

\begin{proof}
A conic $C$ is jumping for $\sE$ if and only if the point of $H$ corresponding to $C$ lies in the support of $R^1q_*p^*(\sE(-1,0))$.\\Let us consider the Fourier-Mukai transform
$\Phi=q_*(p^*(-\otimes\sO_F(-1,0)))$.
Let us apply $\Phi$ to the terms of the monad (\ref{mon2}). $$R^iq_*((p^*(\sO_F(-1,0)\otimes\sO_F(-1,0)))\cong R^iq_*((\sO_F(-2,0)\boxtimes\sO_H(0,0))).$$ By (\ref{uni}) tensored by $\sO_F(-2,0)\boxtimes\sO_H(0,0)$, since the only non zero cohomology on $F$ is $h^2(\sO_F(-3,0))=1$ we get $R^iq_*((p^*(\sO_F(-1,0)\otimes\sO_F(-1,0)))=0$ for $i\not=1$ and
$$R^1q_*((p^*(\sO_F(-1,0)\otimes\sO_F(-1,0)))\cong \sO_{H}(-1,0).$$

$$R^iq_*((p^*(\sO_F(0,-1)\otimes\sO_F(-1,0)))\cong R^iq_*((\sO_F(-1,-1)\boxtimes\sO_{H}(0,0))).$$ By (\ref{uni}) tensored by $\sO_F(-1,-1)\boxtimes\sO_{H}(0,0)$, since the only non zero cohomology on $F$ is $h^3(\sO_F(-2,-2))=1$ we get $R^iq_*((p^*(\sO_F(0,-1)\otimes\sO_F(-1,0)))=0$ for $i\not=1$ and
$$R^1q_*((p^*(\sO_F(0,-1)\otimes\sO_F(-1,0)))\cong \sO_{H}(-1,-1).$$

$$R^iq_*((p^*(\sO_F\otimes\sO_F(-1,0)))\cong R^iq_*((\sO_F(-1,0)\boxtimes\sO_{H}(0,0))).$$ By (\ref{uni}) tensored by $\sO_F(-1,0)\boxtimes\sO_{H}(0,0)$, since the cohomology on $F$ is all zero we get  for  any $i$
$$R^iq_*((p^*(\sO_F\otimes\sO_F(-1,0)))\cong 0.$$

$$R^iq_*((p^*(\sG_1(-1,0)\otimes\sO_F(-1,0)))\cong R^iq_*((\sG_1(-2,0)\boxtimes\sO_{H}(0,0))).$$ By (\ref{uni}) tensored by $\sG_1(-2,0)\boxtimes\sO_{H}(0,0)$, since the only non zero cohomology on $F$ is $h^1(\sG_1(-2,0))=1$ we get $R^iq_*((p^*(\sG_1(-1,0)\otimes\sO_F(-1,0)))=0$ for $i\not=1$ and
$$R^1q_*((p^*(\sG_1(-1,0)\otimes\sO_F(-1,0)))\cong \sO_{H}$$
$$R^iq_*((p^*(\sG_2(0,-1)\otimes\sO_F(-1,0)))\cong R^iq_*((\sG_2(-1,-1)\boxtimes\sO_{H}(0,0))).$$ By (\ref{uni}) tensored by $\sG_2(-1,-1)\boxtimes\sO_{H}(0,0)$, since the only non zero cohomology on $F$ is $h^2(\sG_2(-2,-1))=1$ we get $R^iq_*((p^*(\sO_F(0,1)\otimes\sO_F(-1,0)))=0$ for $i\not=1$ and
$$R^1q_*((p^*(\sG_2(0,-1)\otimes\sO_F(-1,0)))\cong \sO_{H}(-1,0).$$

Now if we apply $\Phi$ to the sequence
$$0 \to K \rightarrow  \sG_1(-1,0)^{\oplus k}\oplus \sG_2(0,-1)^{\oplus k}
\xrightarrow{\beta} \sO_F^{\oplus 2k-2} \to 0$$  we get $R^iq_*((p^*(K\otimes\sO_F(-1,0)))=0$ for $i\not=1$ and
$$R^1q_*((p^*(K\otimes\sO_F(-1,0)))\cong \sO_{H}(-1,0)\oplus\sO_{H}.$$

From
$$0 \to \sO_F(0,-1)^{\oplus k}\oplus \sO_F(-1,0)^{\oplus k} \rightarrow  K
\rightarrow E\to 0$$  we get

$$0 \to R^0q_*((p^*(\sE\otimes\sO_F(-1,0)))\to \sO_{H}(-1,0)^{\oplus k}\oplus\sO_{H}(-1,-1)^{\oplus k} \xrightarrow{\gamma} \sO_{H}(-1,0)^{\oplus k}\oplus\sO_{H}^{\oplus k}.$$ So $\gamma$ is a $(2k)\times (2k)$ matrix made by a $(k)\times (k)$ matrix linear in the first variables, a $(k)\times (k)$ linear in the second variables a $(k)\times (k)$ matrix of degree $0$ and a $(k)\times (k)$ matrix of bidegree $(1,1)$. Hence $\ker(\gamma)$ is zero and
$\coker(\gamma)$, which is $R^1q_*((p^*(\sE\otimes\sO_F(-1,0)))$, is an extension to $H$ of a rank 1 sheaf, which we call $\mathcal{G}$, on $\mathcal D_\sE$, that is a divisor of type $(k,k)$ given by the vanishing of the determinant of $\gamma$.

\end{proof}

\begin{remark}
  An  analogous proof to the one of the previous Proposition was obtain for instanton bundles on $\pp^3$ (see \cite[Theorem 2.2.4]{OSS}) and other Fano threefolds (see \cite[Prop 2.4 and Lemma 4.8]{Fa2}). Moreover, it works also for higher rank instanton bundles on $F$ whose restriction to a general conic is trivial.
\end{remark}

\noindent For properly semistable instanton bundles we can say much more about the divisor of jumping conics. For this recall that such an instanton fits on a short exact sequence (\ref{stabquad}) and in particular has charge $k=l^2$. We are going to see that the support of $\mathcal D_{\sE}$ is exactly the set of \emph{reducible conics}, namely conics of the form $C=L_1\cup L_2$ for lines $L_i$ intersecting transversely on a single point:

\begin{proposition}
  Let $\sE$ be a properly semistable $k$-instanton on $F$, with $k=l^2$. Then $C\in\mathcal D_{\sE}$ if and only if $C$ is reducible: $C=L_1\cup L_2$. In particular, if $\sE$ is defined as the extension
  $$
  0 \to \sO_F(l,-l) \to \sE \to \sO_F(-l,l) \to 0
  $$
  the possible splitting types are given by
  $$
  \sE_{|C}\cong\sO_C(l,-a) \oplus \sO_C(-l,a)
  $$
  with $0\leq a\leq l$.\\
  Moreover, it is possible to find properly semistable instanton bundles $\sE$ on the flag and $l+1$ different reducible conics $C_j$ on $F$, such that
  $$
  \sE_{|C_{j}}\cong\sO_{C_j}(l,-j) \oplus \sO_{C_{ j}}(-l,j),
  $$
  for $j=0,\ldots,l$, i.e. we have all the possible splittings.\\
  The family of properly semistable bundles given by the other extension behaves analogously.
\end{proposition}
\begin{proof}
  Restricting the short exact sequence (\ref{stabquad}) that defines $\sE$ to an smooth conic $C$, we see that $\sE_{|C}\cong\sO_{C}^2$. On the other hand, suppose that $C=L_1\cup L_2$ with each of the $L_i$ belonging to one of the two families of lines on $F$.  Then for the suitable $L_i$ we have
  $$
  0\to\sO_{L_i}(l)\to \sE_{|L_i}\to \sO_{L_i}(-l)\to 0.
  $$
\noindent Then the sequence should split and $\sE_{L_i}\cong\sO_{L_i}(l)\oplus \sO_{L_i}(-l)$. For the other line, the corresponding exact sequence can have central term $\sE_{|L_j}\cong\sO_{L_j}(-a)\oplus \sO_{L_j}(a)$ for any $0\leq a\leq l$. \\

To conclude the proof, consider the semistable family, the other one is studied in the same way, given as extensions of type
$$
0 \rightarrow \sO_F(l,-l) \rightarrow \sE \rightarrow \sO_F(-l,l) \rightarrow 0.
$$
Restricting at the lines of second family, we have that $\sE_{|L_2}\simeq \sO_{L_2}(l) \oplus \sO_{L_2}(-l)$. For the other family of lines we have
$$
0 \rightarrow \sO_{L_1}(-l) \rightarrow \sE_{|L_1} \rightarrow \sO_{L_1}(l) \rightarrow 0,
$$
hence $\sE_{|L_1} \in \Ext^1(\sO_{L_1}(l),\sO_{L_1}(-l)) = H^1(\sO_{L_1}(-2l))$ and $\sE \in \Ext^1(\sO_F(-l,l),\sO_F(-l,-l)) = H^1(\sO_F(2l,-2l))$. The two extensions are related by the following exact sequence
$$
0 \rightarrow \sO_F(2l-2,-2l) \rightarrow \sO_F(2l-1,-2l) \rightarrow \sO_F(2l,-2l) \rightarrow \sO_{L_1}(-2l) \rightarrow 0
$$
From Proposition \ref{pLineBundle} we can compute that $H^i(\sO_F(2l-2,-2l)) =0$ for $i=0,1,2,3$ and $H^i(\sO_F(2l-1,-2l)) = H^i(\sO_F(2l,-2l)) = 0$ for $i=0,2,3$, we therefore have a surjective linear map
\begin{equation}\label{sur-splitt}
\varphi : H^1(\sO_F(2l,-2l)) \rightarrow H^1(\sO_{L_1}(-2l))
\end{equation}

\noindent between vector spaces of respective dimensions $(2l-1)(2l+1)$ and $2l-1$ and let us consider the associated linear projective map $\tilde{\varphi}$ between the associated projective spaces (of dimension one lower). Take lines $L_j$, $j=0,\dots,l$ from this family and points $\alpha_i\in\pp(H^1(\sO_{L_i}(-2l)))$ parameterizing the extension of type $\sO_{L_i}(-i)\oplus\sO_{L_i}(i)$ (recall that the isomorphism class of an extension is classified by this projective space). Then $A_i:=\tilde{\varphi}^{-1}(\alpha_i)$ will be a linear projective subspace of $\pp(H^1(\sO_F(2l,-2l)))$ of codimension $2l-1$. The intersection $\cap_{i=0}^l A_i$ will have codimension $(2l-1)(l+1)\leq (2l-1)(2l+1) -1$ and therefore will be nonempty. The elements of the intersection will parameterize $l^2$-instanton bundles with the desired restriction to the reducible conics $L_i\cup L_i^{'}$ ($L_i^{'}$ is the unique line from the other family that intersects $L_i$).

\end{proof}

Nevertheless, thanks to the following result, we can observe that the divisor of jumping conics of type at least $(-1,1)$ of the generic instanton does not coincide with the family of type at least $(-2,2)$. Hence, we have \emph{proper} jumping conics of type $(-1,1)$. This result should be compared to \cite{BH}.

We denote by $\mathcal S_\sE$ the locus of jumping conics of type at least $(-2,2)$ of an instanton $\sE$, and by $j$ its embedding in $H$.

\begin{proposition}
Let $\sE$ be a generic $k$-instanton on $F$. Then $\mathcal S_{\sE}$ is a Cohen Macaulay curve equipped with a torsion-free sheaf
$G$
fitting into
\begin{equation}\label{jump}0\to B \to\sO_{H}(-1,0)^{\oplus k}\oplus \sO_{H}(0,-1)^{\oplus k}\to  \sO_{H}^{\oplus 2k-2}\to i_*G\to  0.
\end{equation} where $B$ is a rank two torsion free sheaf on $H$.

\end{proposition}

\begin{proof}
A conic $C$ belong to $S_{\sE}$ if and only if the point of $H$ corresponding to $C$ lies in the support of $R^1q_*p^*(\sE)$.\\Let us consider the Fourier-Mukai transform
$\Phi=q_*(p^*())$.
Let us apply $\Phi$ to the terms of the monad (\ref{mon2}). By (\ref{uni}) tensored by $\sO_F(-1,0)\boxtimes\sO_{H}(0,0)$, since the cohomology on $F$ is all zero we get  for  any $i$
$$R^iq_*((p^*(\sO_F))\cong 0.$$

By (\ref{uni}) tensored by $\sO_F(0,-1)\boxtimes\sO_{H}(0,0)$, since the cohomology on $F$ is all zero we get  for  any $i$
$$R^iq_*((p^*(\sO_F\otimes\sO_F(0,-1)))\cong 0.$$

 By (\ref{uni}) tensored by $\sO_F\boxtimes\sO_{H}(0,0)$, since the only non zero cohomology on $F$ is $h^0(\sO_F))=1$ we get $R^iq_*((p^*(\sO_F))=0$ for $i\not=0$ and
$$R^0q_*((p^*(\sO_F))\cong \sO_{H}.$$

 By (\ref{uni}) tensored by $\sG_1(-1,0)\boxtimes\sO_{H}(0,0)$, since the only non zero cohomology on $F$ is $h^1(\sG_1(-2,0))=1$ we get $R^iq_*((p^*(\sG_1(-1,0)))=0$ for $i\not=0$ and
$$R^0q_*((p^*(\sG_1(-1,0)))\cong \sO_{H}(-1,0).$$

 By (\ref{uni}) tensored by $\sG_2(0,-1)\boxtimes\sO_{H}(0,0)$, since the only non zero cohomology on $F$ is $h^2(\sG_2(0,-2))=1$ we get $R^iq_*((p^*(\sO_F(0,1)))=0$ for $i\not=0$ and
$$R^0q_*((p^*(\sG_2(0,-1)))\cong \sO_{H}(0,-1).$$

Now if we apply $\Phi$ to the sequence
$$0 \to \sO(0,-1)^{\oplus k}\oplus \sO(-1,0)^{\oplus k} \rightarrow  K
\rightarrow \sE\to 0$$  we get $R^iq_*((p^*(\sE))\cong R^iq_*((p^*(K))$ for any $i$.  From
$$0 \to K \rightarrow   \sG_1(-1,0)^{\oplus k}\oplus \sG_2(0,-1)^{\oplus k}
\xrightarrow{\beta} \sO^{\oplus 2k-2} \to 0$$  we get
$$0 \to R^0q_*((p^*(\sE))\to \sO_{H}(-1,0)^{\oplus k}\oplus\sO_{H}(0,-1)^{\oplus k} \xrightarrow{\gamma} \sO_{H}^{\oplus (2k-2)}.$$

So $\gamma$ is a $(2k-2)\times (2k)$ matrix made by a $(k-1)\times (k)$ matrix linear in the first variables, a $(k-1)\times (k)$ linear in the seconds variables. Hence $\ker(\gamma)$ is a rank two bundle and
$\coker(\gamma)$, which is $R^1q_*((p^*(\sE))$, is an extension to $H$ of a torsion free sheaf supported on the curve $\mathcal S_\sE$.

\end{proof}

\bibliographystyle{amsplain}

\end{document}